\documentclass[a4paper, 11pt, twoside]{article}
\usepackage[margin=0.8in, top=1in]{geometry}
\usepackage{amssymb, amsmath, amsthm, amsfonts, thmtools, yhmath, bm, stmaryrd}
\usepackage{enumerate, enumitem}
\usepackage{mathtools,microtype,hyperref,xstring,braket}
\usepackage[noabbrev,capitalise]{cleveref}
\usepackage{scalerel}
\usepackage{graphicx, float, color}
\usepackage[dvipsnames]{xcolor}
\usepackage[framemethod=tikz]{mdframed}
\usepackage{titling}
\usepackage{subcaption}
\usepackage[export]{adjustbox}
\usepackage{tikz-cd, tikz-3dplot}
\usetikzlibrary{calc, math, fit, positioning}
\usepackage{faktor}

\usepackage{fancyhdr}
\usepackage[parfill]{parskip} 
\pretitle{\begin{center}\Huge}
\posttitle{\end{center}}
\preauthor{\begin{center}\large}
\postauthor{\end{center}}
\predate{\begin{center}\large}
\postdate{\end{center}\vspace{-1ex -1em}}
\usepackage{cellspace}
\usepackage{booktabs}

\AtBeginEnvironment{proof}{\footnotesize} %make proofs small

\usepackage[backend=biber,style=alphabetic,sorting=ynt]{biblatex}
\defbibheading{bibliography}[\bibname]{\section*{#1}}

\usepackage{stackengine}

\usepackage{textcomp,gensymb}

\mdfdefinestyle{mdyellowbox}{%Proposition style
    skipabove=8pt,
    linewidth=2pt,
    rightline=false,
    leftline=true,
    topline=false,
    bottomline=false,
    linecolor=Dandelion,
    linewidth=0.05in,
}
\declaretheoremstyle[
    headfont=\bfseries\color{black},
    bodyfont=\normalfont,
    spaceabove=2pt,
    spacebelow=1pt,
    mdframed={style=mdyellowbox},
    headpunct={ --- }
]{thmyellowbox}

\mdfdefinestyle{mdlilacbox}{%Corollary style
    skipabove=8pt,
    linewidth=2pt,
    rightline=false,
    leftline=true,
    topline=false,
    bottomline=false,
    linecolor=Orchid,
    linewidth=0.05in,
}
\declaretheoremstyle[
    headfont=\bfseries\color{black},
    bodyfont=\normalfont,
    spaceabove=2pt,
    spacebelow=1pt,
    mdframed={style=mdlilacbox},
    headpunct={ --- }
]{thmlilacbox}

\mdfdefinestyle{mdbluebox}{%Lemma style
    skipabove=8pt,
    linewidth=2pt,
    rightline=false,
    leftline=true,
    topline=false,
    bottomline=false,
    linecolor=cyan, % I prefer 3C6E71
    linewidth=0.05in,
}
\declaretheoremstyle[
    headfont=\bfseries\color{black},
    bodyfont=\normalfont,
    spaceabove=2pt,
    spacebelow=1pt,
    mdframed={style=mdbluebox},
    headpunct={ --- }
]{thmbluebox}

\mdfdefinestyle{mdgreenbox}{%Remark style
    skipabove=8pt,
    linewidth=2pt,
    rightline=false,
    leftline=true,
    topline=false,
    bottomline=false,
    linecolor=ForestGreen, % I prefer BEEE62 or 70AE6E for a darker one
    linewidth=0.05in,
}
\declaretheoremstyle[
    headfont=\bfseries\color{black},
    bodyfont=\normalfont,
    spaceabove=2pt,
    spacebelow=1pt,
    mdframed={style=mdgreenbox},
    headpunct={ --- }
]{thmgreenbox}

\mdfdefinestyle{mdredbox}{%Theorem style
    skipabove=8pt,
    linewidth=2pt,
    rightline=false,
    leftline=true,
    topline=false,
    bottomline=false,
    linecolor=red,
    linewidth=0.05in,
}

\declaretheoremstyle[
    headfont=\bfseries\color{black},
    bodyfont=\normalfont,
    spaceabove=2pt,
    spacebelow=1pt,
    mdframed={style=mdredbox},
    headpunct={ --- }
]{thmredbox}

\mdfdefinestyle{mdbrownbox}{%Theorem style
    skipabove=8pt,
    linewidth=2pt,
    rightline=false,
    leftline=true,
    topline=false,
    bottomline=false,
    linecolor=brown,
    linewidth=0.05in,
}

\declaretheoremstyle[
    headfont=\bfseries\color{black},
    bodyfont=\normalfont,
    spaceabove=2pt,
    spacebelow=1pt,
    mdframed={style=mdbrownbox},
    headpunct={ --- }
]{thmbrownbox}

\declaretheoremstyle[
    headfont=\bfseries\color{black},
    bodyfont=\normalfont,
    spaceabove=2pt,
    spacebelow=1pt,
    mdframed={style=mdgraybox},
    headpunct={ --- }
]{thmgraybox}

\mdfdefinestyle{mdgraybox}{%Theorem style
    skipabove=8pt,
    linewidth=2pt,
    rightline=false,
    leftline=true,
    topline=false,
    bottomline=false,
    linecolor=gray,
    linewidth=0.05in,
}

\theoremstyle{definition}
\declaretheorem[name=Theorem,sibling=theorem,style=thmredbox,numberwithin=section]{theorem}
\declaretheorem[name=Proposition,sibling=theorem,style=thmyellowbox]{proposition}
\declaretheorem[name=Lemma,sibling=theorem,style=thmbluebox]{lemma}

\declaretheorem[name=Corollary,sibling=theorem,style=thmlilacbox]{corollary}

\declaretheorem[name=Conjecture,sibling=theorem,style=thmbrownbox]{conjecture}

\renewcommand{\a}{\mathfrak{a}}

\renewcommand{\b}{\mathfrak{b}}

\renewcommand{\c}{\mathfrak{c}}

\newcommand{\G}{\mathbf{G}}
\newcommand{\g}{\mathfrak{g}}

\newcommand{\h}{\mathfrak{h}}

\renewcommand{\L}{\mathcal{L}}
\newcommand{\N}{\mathbb{N}}
\newcommand{\n}{\mathfrak{n}}

\newcommand{\m}{\mathfrak{m}}

\newcommand{\p}{\mathfrak{p}}

\newcommand{\Q}{\mathbb{Q}}

\newcommand{\R}{\mathbb{R}}

\newcommand{\s}{\mathfrak{s}}

\renewcommand{\t}{\mathfrak{t}}

\newcommand{\V}{\mathcal{V}}

\newcommand{\x}{\textnormal{\textbf{x}}}

\newcommand{\Z}{\mathbb{Z}}

\DeclareMathOperator{\Ann}{Ann}
\DeclareMathOperator{\Aut}{Aut}

\newcommand{\rad}{\textnormal{rad}}
\newcommand{\del}{\partial}

\DeclareMathOperator{\GL}{GL}
\DeclareMathOperator{\gr}{gr}

\DeclareMathOperator{\id}{id}

\renewcommand{\sl}{\mathfrak{sl}}
\renewcommand{\sp}{\mathfrak{sp}}
\DeclareMathOperator{\SP}{SP}
\DeclareMathOperator{\tr}{tr}

\newcommand{\wh}{\widehat}
\DeclareMathSymbol{\shortminus}{\mathbin}{AMSa}{"39} 
\newcommand{\norm}[1]{\left\lVert #1 \right\rVert}
\newcommand{\abs}[1]{\left\lvert #1 \right\rvert}

\newcommand{\aff}[2][]{\wh{U ( \mathfrak{#2}_{#1})}_K}

\newcommand{\hyp}[2][]{U ( \mathfrak{#2}_{#1})_K}

\newcommand{\Weyl}[1][n]{\wh{A_{#1}(R)}_K}

\newcommand\restr[2]{{% we make the whole thing an ordinary symbol
  \left.\kern-\nulldelimiterspace % automatically resize the bar with \right
  #1 % the function
  \littletaller % pretend it's a little taller at normal size
  \right|_{#2} % this is the delimiter
  }}
\newcommand{\littletaller}{\mathchoice{\vphantom{\big|}}{}{}{}}
\makeatletter
\newcommand{\smallbullet}{} % for safety
\DeclareRobustCommand\smallbullet{%
  \mathord{\mathpalette\smallbullet@{0.7}}%
}
\newcommand{\smallbullet@}[2]{%
  \vcenter{\hbox{\scalebox{#2}{$\,\m@th#1\bullet$}}}%
}
\makeatother
\newcommand\blfootnote[1]{% Blank footnote
  \begingroup
  \renewcommand\thefootnote{}\footnote{#1}%
  \addtocounter{footnote}{-1}%
  \endgroup
}

\begin{document}

\title{The Metaplectic Representation is Faithful}
\date{May 2025}
\author{Christopher Chang\thanks{St Anne's College, University of Oxford}, Simeon Hellsten\thanks{St John's College, University of Oxford}, Mario Marcos Losada\thanks{Brasenose College, University of Oxford} and Sergiu Novac\protect\footnotemark[2]}
\maketitle
\begin{abstract}
    \noindent  We develop methods to show that infinite-dimensional modules over the Iwasawa algebra $KG$ of a uniform pro-$p$ group are faithful and apply them to show that the metaplectic representation for the group $G=\exp(p\, \sp_{2n}(\Z_p))$ is faithful. \blfootnote{\textit{Email addresses}: \href{mailto:christopher.chang@st-annes.ox.ac.uk}{\texttt{christopher.chang@st-annes.ox.ac.uk}} (C. Chang), \href{mailto:simeon.hellsten@sjc.ox.ac.uk}{\texttt{simeon.hellsten@sjc.ox.ac.uk}} (S. Hellsten), \href{mailto:mario.marcoslosada@bnc.ox.ac.uk}{\texttt{mario.marcoslosada@bnc.ox.ac.uk}} (M. Marcos Losada), \href{mailto:sergiu-ionut.novac@sjc.ox.ac.uk}{\texttt{sergiu-ionut.novac@sjc.ox.ac.uk}} (S. Novac)} 
\end{abstract}

\section{Introduction}
Let $p$ be an odd prime and $G$ be a uniform pro-$p$ group. Let $K$ be a finite extension of $\Q_p$ with valuation ring $R\coloneqq\{x\in K: |x|_p\le 1\}$. We are interested in the prime ideals of the Iwasawa algebra $KG\coloneqq RG\otimes_R K$ where $RG\coloneqq\varprojlim R[G/N]$ and the inverse limit is taken over all open normal subgroups $N\lhd_o G$. If $G$ has a closed normal subgroup $N\lhd_c G$ such that $G/N$ is uniform, then $(N-1)KG=\ker(KG\to K(G/N))$ is a prime ideal of $KG$ so we will restrict to almost simple groups $G$. We are motivated by the following conjecture.

\begin{conjecture}\label{Conj: Ideal structure}
    Let $G$ be an almost simple uniform pro-$p$ group. Then every non-zero prime ideal of $KG$ has finite codimension.
\end{conjecture}
In particular, if the conjecture is true, then every infinite-dimensional representation of $KG$ must be faithful. When the Lie algebra of $G$ is of Type $A$, \cref{Conj: Ideal structure} was proven in \cite{manndiss}, but the methods there do not generalise to other types. Similar results of this form can also be found for example in \cite{ardakov2013verma}. In this paper, we develop methods to approach these types of questions more generally and we apply them to the case of $G\coloneqq \exp(p\g),$ where $\g\coloneqq\sp_{2n}(\Z_p).$
We start with the metaplectic representation from \cref{Prop: Metaplectic}
and explain how to lift this to an algebra homomorphism $\rho:\aff{g} \to \Weyl$, where $\aff{g}$ is the affinoid enveloping algebra and $\Weyl$ is the completed Weyl algebra, both defined after \cref{Prop: Metaplectic}. We can finally embed $KG$ into $\aff{g}$ and in \cref{Thm: Main} we then prove the following.

\begin{theorem}
    The metaplectic representation $\restr{\rho}{KG}:KG\to\Weyl$ is injective.
\end{theorem}
The main tool to do this is the following Gluing Lemma, a special case of the general version from \cref{Sec: Gluing lemma}. Here, $U(\n)_K=U(\n)\otimes_{\Z_p} K$.

\begin{proposition}\label{Prop: Gluing I}
    Let $\n, \h$ and $\g\coloneqq\n\oplus\h$ be finite rank $\Z_p$-Lie algebras with corresponding uniform pro-$p$ groups $N,H$ and $G$ and let $V$ be a $\aff{g}$-module and a $K-$Banach space, such that $\left( N-1 \right) \cdot B_V \subseteq pB_V,$ where $B_V\coloneqq\{v\in V: \norm{v}\le 1\}$. Suppose there is a subset $\V$ of $V$ and for each $v\in\V$ a $\aff{h}$-submodule $W_v$ of $V$ contained in $\hyp{n}\cdot v$ such that
    \begin{itemize}
        \item $RH$ acts locally finitely on $W\coloneqq \sum_{v\in\V} W_v$, meaning that every cyclic $RH$-submodule of $W$ is finitely generated over $R$.
        \item $KH$ acts faithfully on $W$.
        \item For every $v\in\V$ the multiplication map $KN\otimes_K \faktor{\hyp{n}}{I_v'}\to \faktor{\aff{n}}{I_v}$ is injective, where $I_v'=\Ann_{\hyp{n}}(v)$ and $I_v=\Ann_{\aff{n}}(v)$.
    \end{itemize}
    Then any $KG$-submodule $V_0$ of $V$ containing $W$ is faithful.
\end{proposition}

This allows us to deduce faithfulness results for a Lie algebra $\g=\n\oplus\h$ from similar results for $\n$ and $\h$. We do this by exploiting the local finiteness conditions to replace the action of $\h$ by an action of $\hyp{n}$, thus turning the problem into a question involving only the Lie algebra $\n$.

In \cref{Sec: Metaplectic} we then apply this to the metaplectic representation. We first decompose $\g=\a\oplus\b\oplus\c$ into subalgebras $\a, \b$ and $\c$ with associated uniform pro-$p$ groups $A, B$ and $C$, where $\a$ acts locally finitely and $\b$ acts locally nilpotently, and proceed in three steps. 
\begin{itemize}
    \item In \cref{Subsec: Multiplication map} we start by generalising \cite[Theorem 3.8]{ardakov2013verma} to show that the multiplication map from the last condition of the Gluing Lemma is indeed injective for any subalgebra of $\c$. To do this we use the generalisation of the Gluing Lemma from \cref{Sec: Gluing lemma}.
    \item In \cref{Subsec: KA} we show that the second condition for the Gluing Lemma above holds, namely that the metaplectic representation is injective when restricted to $KA$. This involves exploiting symmetries of $\Weyl$ to be able to apply the Gluing Lemma as stated above. 
    \item In \cref{Subsec: Gluing} we put together the previous results to conclude that the metaplectic representation is injective. 
\end{itemize}

In \cref{Sec: Abelian subalgebras} we use primary decomposition and a result from \cite{ardakov2012prime} on ideals fixed by a particular class of action to give a different proof for abelian subalgebras of general Iwasawa algebras. We obtain the following result, where $\L$ denotes the $\Q_p$-Lie algebra associated to a uniform pro-$p$ group, as in \cite[\textsection 9.5]{ddms2003analytic}.
\begin{theorem}
    Let $G$ be a uniform pro-$p$ group, and $H \leq G$ a torsion-free abelian pro-$p$ subgroup, with associated $\Z_p$-Lie algebras $\g,\h$. Let $\n \subseteq \g$ be a subalgebra contained in the normaliser of $\h$, $N$ its associated pro-$p$ group, and assume $\g/\n$ is torsion-free. 
    
    Suppose $T$ is a filtered $K$-algebra, and $\psi: KG \to T$ is a filtered $K$-algebra homomorphism such that $\psi(KH)$ is not finite $K$-dimensional. If $\L(H)$ is an irreducible $\L(N)$-module, then $\restr{\psi}{KH}$ is injective.
\end{theorem}
In particular, this shows the faithfulness of the metaplectic representation when restricted to the abelian subalgebras $KB$ and $KC$.

The authors believe that the methods developed here can be generalised to an arbitrary highest weight module for $\sp_{2n}$ and more generally to other simple Lie algebras. 

\subsection{Acknowledgements}
The authors thank Konstantin Ardakov for his invaluable support, advice and feedback throughout this project. The first author thanks St Anne's College, Oxford for providing accommodation and funding for living expenses for the duration of this project. The second and last authors thank St John's College, Oxford for providing additional funding and accommodation for the duration of the project. The third author thanks the London Mathematical Society and the Mathematical Institute for funding this project. The last author also thanks the Mathematical Institute for funding the project.

\section{Constructing the Metaplectic Representation}\label{Sec: Introduction}
We will use lower-case letters to denote elements in a $\Z_p$-Lie algebra $\g$ and the corresponding upper-case letter for the corresponding element in its associated uniform pro-$p$ group $G\coloneqq\exp(p \g)$. It is a standard fact of Iwasawa algebras (see for example \cite[\textsection 7]{ddms2003analytic}) that a general element $\zeta\in KG$ can be uniquely written as \begin{equation*}
    \zeta = \sum_{\alpha \in \N_0^d}\lambda_\alpha (\G-1)^\alpha, \qquad \qquad (\G-1)^\alpha \coloneqq (G_1-1)^{\alpha_1}\cdots (G_d-1)^{\alpha_d},
\end{equation*} where $\left(G_1,\dots,G_d\right)$ is a topological generating set for $G$ and $\lambda_{\alpha} \in K$ are uniformly bounded with respect to the $p$-adic valuation on $K$. 

\noindent For fixed $n\ge 2$, let $$\g \coloneqq  \sp_{2n}(\Z_p)=\left\{\begin{bmatrix}
        A & B \\
        C & -A^T
    \end{bmatrix}: A, B, C\in M_n(\Z_p), B^T=B, C^T=C\right\}$$ with associated uniform pro-$p$ group $$G\coloneqq\exp(p\, \sp_{2n}(\Z_p))=\SP_{2n}(\Z_p)\cap(I+pM_{2n}(\Z_p)).$$ 
    We will use the following map adapted from \cite{folland1989harmonic} to construct an infinite dimensional $KG$-module.

\begin{proposition}[Metaplectic Representation]\label{Prop: Metaplectic}
    There is a Lie algebra homomorphism $\sp_{2n}(\Z_p)\to A_n(\Z_p)$ given by
    $$\begin{bmatrix}
        A & B \\
        C & -A^T
    \end{bmatrix}\longmapsto-\frac{1}{2}\tr(A)-\sum_{1\leq i,j\leq n}{A_{ij}x_j\del_i}+\frac{1}{2}\sum_{1\leq i,j\leq n}{B_{ij}\del_i\del_j}-\frac{1}{2}\sum_{1\leq i,j\leq n}{C_{ij}x_ix_j}.$$
\end{proposition}
Here $A_n(\Z_p)$ is the $n^{th}$ Weyl algebra on generators $x_1, \ldots, x_n$ and $\del_1,\ldots, \del_n$. We will denote $\x^{\alpha}\coloneqq x_1^{\alpha_1}\cdots x_n^{\alpha_n}$ and $\del^{\beta}\coloneqq\del_1^{\beta_1}\cdots\del_n^{\beta_n}$ for all $\alpha,\beta \in \N_0^n$.

We are interested in knowing if there are any non-zero prime ideals of $KG$ of infinite codimension. To this end, we lift the above homomorphism to a map on $KG.$ We tensor first with $R$ to obtain a map of $R$-Lie algebras $\g_R\coloneqq\sp_{2n}(R)\to A_n(R)$ and lift this to a map of associative algebras $U(\g_R)\to A_n(R),$ where $U(\g_R)$ is the universal enveloping algebra.  

Now note that both $U(\g_R)$ and $A_n(R)$ inherit $p$-adic valuations from $R$. Explicitly, as $K$ is a finite $\Q_p$ extension, the valuation on $\Q_p$ can be extended to a valuation on $K,$ which we denote by $v$. Let $d=2n^2+n$ be the dimension of $\g$ over $\Z_p$ and fix a basis $g_1,\ldots, g_d$ of $\g$. We denote $\mathbf{g}^{\alpha}\coloneqq g_1^{\alpha_1}\cdots g_d^{\alpha_d}$ for $\alpha \in \N_0^d$, and write $\abs{\alpha} = \alpha_1 + \cdots + \alpha_n$. Then, for an element of $U(\g_R)$ we define its valuation by
$$v_p\left(\sum_{\alpha\in\N_0^d}c_{\alpha}\mathbf{g}^{\alpha}\right)\coloneqq\min_{\alpha\in\N_0^d}v(c_{\alpha}),$$ and similarly, for elements of $A_n(R),$
$$v_p\left(\sum_{\alpha,\beta\in\N_0^n}c_{\alpha,\beta}\x^{\alpha}\del^{\beta}\right)\coloneqq\min_{\alpha,\beta\in\N_0^n}v(c_{\alpha,\beta}).$$

We then note the algebra homomorphism is continuous with respect to the induced topologies. Therefore, this extends to a map $\wh{U(\g_R)}\to \wh{A_n(R)}$ of the $p$-adic completions. 

Then $$\wh{U(\g_R)}\coloneqq\varprojlim_{\lambda\geq 0}\faktor{U(\g_R)}{U(\g_R)_{\lambda}}=\left\{\sum_{\alpha\in \N_0^d}c_{\alpha}\mathbf{g}^{\alpha} : c_{\alpha} \in R, c_{\alpha}\to 0 \text{ as } |\alpha| \to \infty \right\}$$

with $U(\g_R)_\lambda\coloneqq\left\{x\in U(\g_R): v_p(x)\geq \lambda \right\}$ for $\lambda\in\R$, and similarly $$\wh{A_n(R)}\coloneqq\varprojlim_{\lambda \geq 0}\faktor{A_n(R)}{A_n(R)_{\lambda}}=\left\{ \sum_{\alpha,\beta\in \N_0^n}c_{\alpha,\beta}\x^{\alpha}\del^{\beta}: c_{\alpha,\beta}\in R, c_{\alpha,\beta}\to 0 \text{ as } |\alpha|+|\beta| \to \infty  \right\},$$ with $A_n(R)_{\lambda}\coloneqq\left\{y\in A_n(R): v_p(y)\geq \lambda \right\}$ for $\lambda\in\R$.

Now, we tensor with $K$ and obtain a map $\rho:\aff{g}\to \Weyl$, where $$\aff{g}\coloneqq\wh{U(\g_R)}\otimes_R K=\left\{\sum_{\alpha\in \N_0^n}c_{\alpha}\mathbf{g}^{\alpha}: c_{\alpha}\in K, c_{\alpha}\to 0 \text{ as } |\alpha| \to \infty \right\}$$ and $$\Weyl\coloneqq\wh{A_n(R)}\otimes_R K=\left\{ \sum_{\alpha,\beta\in \N_0^n}c_{\alpha,\beta}\x^{\alpha}\del^{\beta}:  c_{\alpha,\beta}\in K, c_{\alpha,\beta}\to 0 \text{ as } |\alpha|+|\beta| \to \infty  \right\}.$$
We also define $\hyp{g}\coloneqq U(\g_R)\otimes_R K\subseteq \aff{g}$. Finally, as noted in \cite[Corollary 2.5.4]{stephen2022TypeA} $KG$ embeds into $\aff{g}$ via $g\mapsto e^{pg}$ for $g\in G$, so we can restrict $\rho$ along this embedding to obtain a map $\rho|_{KG}:KG\to \Weyl$. The aim is to show that $\ker\restr{\rho}{KG}=0$.

Let us also mention the valuations defined on $U(\g_R)$ and $A_n(R)$ extend in the natural way to valuations on $\aff{g},$ and $\Weyl$ respectively. That these are indeed valuations follows from \cite[Lemma 5.2]{ardakov2013verma} together with \cite[Remark 4.6]{PeterSchneider} for $\aff{g}$ and \cite[Lemma 1.2.4]{pangalosdiss} for the Weyl algebra (and by continuity for the completed Weyl algebra as well). 

From now on, we fix the following basis for $\g$. Let $e_{ij}=[\delta_{iI}\delta_{jJ}]_{IJ}$ for $1\le i, j\le 2n$ denote the $2n\times 2n$ unit matrices. Then, for $1\le i, j\le n$ we let 
\begin{align*}
    a_{ij}&=e_{ij}-e_{j+n, i+n} & b_{ij}&=e_{i, j+n}+e_{j, i+n} & c_{ij}&=e_{i+n, j}+e_{j+n, i} \\
    \rho(a_{ij})&=-\frac{1}{2}\delta_{ij}-x_j\del_i & \rho(b_{ij})&=\del_i\del_j & \rho(c_{ij})&=-x_ix_j
\end{align*}
so note that in particular $b_{ij}=b_{ji}$ and $c_{ij}=c_{ji}$. We record here the commutation relations for later reference.

\begin{lemma}\label{Lmm: Commutators}
    For all $1\le i, j, k, l\le n$ we have 
        $$[a_{ij}, a_{kl}]=\delta_{jk}a_{il}-\delta_{il}a_{kj}, \qquad [a_{ij}, b_{kl}]=\delta_{jk}b_{il}+\delta_{jl}b_{ik}, \qquad [a_{ij}, c_{kl}]=-\delta_{il}c_{jk}-\delta_{ik}c_{jl},$$
        $$[b_{ij}, c_{kl}] = \delta_{jk}a_{il}+\delta_{jl}a_{ik}+\delta_{ik}a_{jl}+\delta_{il}a_{jk}, \qquad [b_{ij}, b_{kl}]=0, \qquad [c_{ij}, c_{kl}]=0.$$ 
\end{lemma}
\begin{proof}
    For all $1\le i, j, k, l\le 2n$ we have $$[e_{ij}, e_{kl}]=\delta_{jk}e_{il}-\delta_{il}e_{kj}$$ which gives the relations above.
\end{proof}

In particular, we have the subalgebras
$$\a=\langle a_{ij}: 1\le i, j\le n\rangle_{\Z_p} \qquad \b=\langle b_{ij}: 1\le i\le j\le n\rangle_{\Z_p} \qquad \c=\langle c_{ij}: 1\le i\le j\le n\rangle_{\Z_p}$$ and denote by $A, B$ and $C$ the corresponding uniform pro-$p$ groups. We present the root space decompositions of $\sp_4$ and of $\sp_6$ as illustrative examples. In particular, it is easy to see from these that $\b$ and $\c$ are abelian. 

\tdplotsetmaincoords{62}{13}
\begin{figure}[!htb] %FIX BOUNDING BOXES
    \centering
    \includegraphics[page=1,scale=1, valign=c]{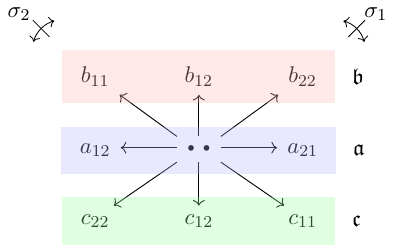} \qquad \qquad
    \includegraphics[page=2,scale=1, valign=c]{images.pdf}
\caption{Root space decomposition of $\sp_4$, and image under $\rho$}
\end{figure}

\begin{figure}[!htb]
    \centering
    \includegraphics[page=3, scale=0.74, valign=c]{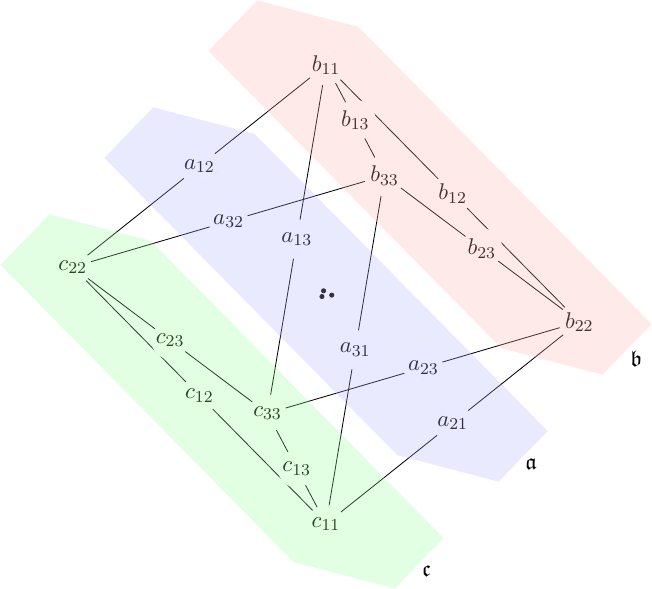}
    \includegraphics[page=4, scale=0.74, valign=c]{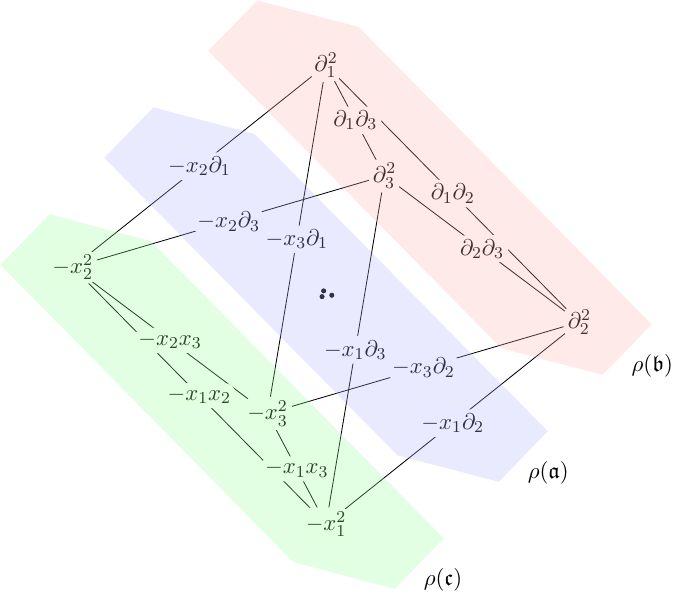}
    \caption{Root space decomposition of $\sp_6$, and image under $\rho$}
\end{figure}

In order to translate results between the different subalgebras of $\g$ we will make use of the following Fourier transforms of $\Weyl$. For $1\le i,j\le n$, let $\tau_i: \Weyl\to \Weyl$ be the continuous extension of the automorphism on $A_n(R)_K$ given by 
\begin{align*} \tau_i(x_j) = \begin{cases} \del_j & \text{ if } i =j\\ x_j & \text{ if } i \neq j \end{cases} && \tau_i(\del_j) = \begin{cases} -x_j & \text{ if } i =j\\ \del_j & \text{ if } i \neq j, \end{cases}
\end{align*}
and note that $\tau_i$ is still an automorphism.

\begin{lemma}\label{Lmm: Weyl automorphisms}
    For every $1\le i\le n$, there is an automorphism $\sigma_i:\aff{g}\to\aff{g}$ making the following diagram commute:

    \[\begin{tikzcd}
	\aff{g} && \aff{g} \\ 
	\\
	{\wh{A_n(R)}}_K && {\wh{A_n(R)}}_K
	\arrow["\rho"', from=1-1, to=3-1]
	\arrow["\rho", from=1-3, to=3-3]
	\arrow["\sigma_i", from=1-1, to=1-3]
	\arrow["\tau_i", from=3-1, to=3-3]
\end{tikzcd}\]
    Explicitly, $$\sigma_i(a_{jk})=\begin{cases}
        -a_{ii} & \text{ if } j=k=i \\
        -c_{ik} & \text{ if } j=i, k\neq i \\
        -b_{ji} & \text{ if } j\neq i, k=i \\
        a_{jk} & \text{ otherwise }
    \end{cases}\qquad
    \sigma_i(b_{jk})=\begin{cases}
        -c_{ii} & \text{ if } j=k=i \\
        a_{ki} & \text{ if } j=i, k\neq i \\
        a_{ji} & \text{ if } j\neq i, k=i \\
        b_{jk} & \text{ otherwise }
    \end{cases}$$
    $$\sigma_i(c_{jk})=\begin{cases}
        -b_{ii} & \text{ if } j=k=i \\
        a_{ik} & \text{ if } j=i, k\neq i \\
        a_{ij} & \text{ if } j\neq i, k=i \\
        c_{jk} & \text{ otherwise. }
    \end{cases}
    $$
\end{lemma} 
\begin{proof}
    From \cref{Prop: Metaplectic}, $\tau_i$ preserves the image $\rho(\g)$, and $\restr{\rho}{\g}$ is injective, so we can pull $\restr{\tau_i}{A_n(R)}$ back to an automorphism of $\g$ and complete it to get an automorphism $\sigma_i$ of $\aff{g}$ that makes the above diagram commute. Since $\rho(\sigma_i(x))=\tau_i(\rho(x))$ for all $x\in\g$ we obtain the explicit formulas above. 
\end{proof}

We also let the total Fourier transform on $\Weyl$ be $\tau\coloneqq\tau_1\circ\cdots\circ\tau_n$ with the corresponding automorphism on $\aff{\g}$ be $\sigma\coloneqq\sigma_1\circ\cdots\circ\sigma_n$. 

\section{The Gluing Lemma}\label{Sec: Gluing lemma}
In this section we develop the main tool for proving the injectivity of $\restr{\rho}{KG}$, partly reducing the problem to showing it is injective when restricted to various subalgebras of $KG$. Note that setting $T = R$ gives \cref{Prop: Gluing I}.

\begin{proposition}[Gluing Lemma] \label{Prop: Gluing II}
    Let 
    \begin{itemize}
        \item $\n, \h$ and $\g\coloneqq\n\oplus\h$ be finite rank $\Z_p$-Lie algebras
    with corresponding uniform pro-$p$ groups $N,H$ and $G$.
        \item $T=\bigcup_{d\in\Z} F_dT$ be an associative $R$-algebra with a $\Z$-filtration by $R$-modules, and $T_K\coloneqq T\otimes_R K$.
        \item $V$ be a $\aff{g}\otimes_K T_K$-module and a $K$-Banach space, such that $\left( N-1 \right) \cdot B_V \subseteq pB_V,$ where $B_V\coloneqq\{v\in V: \norm{v}\le 1\}$.
        \item $\V$ be a subset of $V$ and for each $v\in\V$, let $W_v$ be a $\aff{h}\otimes_K T_K$-submodule of $V$ contained in $\hyp{n}\cdot v$.
        \item $V_0$ be a $KG\otimes_K T_K$-submodule of $V$ containing $W\coloneqq\sum_{v\in\V}W_v$.
    \end{itemize}
    Suppose the following conditions hold:
    \begin{itemize}
        \item For each $d\in\Z$ and $w\in W$, $RH\otimes_R F_dT\cdot w$ is finitely generated over $R$.
        \item $KH\otimes_K T_K$ acts faithfully on $W$.
        \item For every $v\in\V$, the multiplication map $$KN\otimes_K \faktor{\hyp{n}}{I_v'}\to \faktor{\aff{n}}{I_v}$$ is injective, where $I_v' \coloneqq \Ann_{\hyp{n}}(v)$ and $I_v\coloneqq \Ann_{\aff{n}}(v)$.
    \end{itemize}
    Then $V_0$ is a faithful $KG\otimes_K T_K$-module.
\end{proposition}
\begin{proof}
    Let $(n_1, \ldots, n_k)$ be a basis for $\n$ and $(N_1, \ldots, N_k)$ the corresponding topological generating set for $N$. Note that it is enough to show that $V_0$ is faithful as an $RG\otimes_R T$-module. Take $\zeta \in \Ann_{RG\otimes_R T}(V_0)$ and note that we can write it as $\zeta=\sum_{\alpha \in \N_0^k}\left(\mathbf{N}-1 \right)^{\alpha}\zeta_{\alpha}$ where $\zeta_{\alpha}\in RH\otimes_R F_dT$ for some $d\in\Z$ which does not depend on $\alpha$.
    Fix $v \in \V$ and $w \in W_v$. Now since $$RH\otimes_R F_dT\cdot w\subseteq \hyp{n}\cdot v$$ is finitely generated over $R$, we can choose $u_1, \ldots, u_r\in U(\n)$ such that $$RH\otimes_R F_dT\cdot w\subseteq \langle u_i\cdot v: 1\le i\le r\rangle_K$$ and $u_1+I_v', \ldots, u_r+I_v'$ are $K$-linearly independent in $\faktor{\hyp{\n}}{I_v'}\cong \hyp{\n}\cdot v$.  Moreover, by taking common denominators for the coefficients of $u_i\cdot v$ in a finite generating set for $RH\otimes_R F_dT\cdot w$ over $R$, we see that $$RH\otimes_R F_dT\cdot w\subseteq p^{-m}\langle u_i\cdot v: 1\le i\le r\rangle_R$$ for some $m > 0$ and so we can write $$\zeta_{\alpha}\cdot w=\sum_{i=1}^r\lambda_{i, \alpha} u_i\cdot v$$
    for some $\lambda_{i, \alpha}\in K$ which are uniformly bounded in $\alpha$ and $i$. The condition $\left(N-1 \right)\cdot B_V \subseteq p B_V$ then ensures  
    $$0=\zeta\cdot w=\left[\sum_{1 \leq i \leq r}\sum_{\alpha\in\N_0^k}\lambda_{i, \alpha}(\mathbf{N}-1)^{\alpha}u_i\right]\cdot v$$ so $$\sum_{1 \leq i \leq r}\sum_{\alpha\in\N_0^k}\lambda_{i, \alpha}(\mathbf{N}-1)^{\alpha}u_i\in I_v$$ and by the injectivity of the multiplication map we have $$\sum_{1 \leq i \leq r}\left[\sum_{\alpha\in\N_0^k}\lambda_{i, \alpha}(\mathbf{N}-1)^{\alpha}\right]\otimes_K(u_i+I_v')=0.$$ Since $u_i+I_v'$ are linearly independent over $K$, this can only be the case if $$\sum_{\alpha\in\N_0^k}\lambda_{i, \alpha}(\mathbf{N}-1)^{\alpha}=0$$ in $KN$ for all $1\le i\le r$. But then $\lambda_{i, \alpha}=0$ so $\zeta_{\alpha}\cdot w=0$. By linearity this is true for all $w\in W$ and since $KH\otimes_K T_K$ acts faithfully on $W$ we get that $\zeta_{\alpha}=0$ for all $\alpha\in\N_0^k$. Then $\zeta=0$, $V_0$ is faithful as an $RG \otimes_R T$-module, and so also as a $KG\otimes_K T_K$-module.
\end{proof}

\section{Faithfulness of the Metaplectic Representation}\label{Sec: Metaplectic}

Our next goal is to show that $\restr{\rho}{KG}: KG\to \Weyl$ is injective. In order to do this, we will repeatedly apply \cref{Prop: Gluing II}  with the $K$-algebras

$$V\coloneqq K\langle X_1^{\pm}, \ldots, X_n^{\pm}\rangle=\left\{\sum_{\alpha\in\Z^n}\lambda_{\alpha} X_1^{\alpha_1} \cdots X_n^{\alpha_n}: \lambda_\alpha \in K,  \lambda_{\alpha}\to 0 \text{ as } \abs{\alpha}\to \infty\right\},$$

$$V_0\coloneqq K\langle X_1, \ldots, X_n\rangle=\left\{\sum_{\alpha\in\N_0^n}\lambda_{\alpha} X_1^{\alpha_1} \cdots X_n^{\alpha_n}: \lambda_\alpha \in K, \lambda_{\alpha}\to 0 \text{ as } \abs{\alpha}\to \infty\right\}$$
which are naturally also $\Weyl$-modules, and so also $\aff{g}$-modules along the map $\rho:\aff{g} \to \Weyl$. Here,  $|\alpha|\coloneqq |\alpha_1|+|\alpha_2|+...+|\alpha_n|$ for $\alpha \in \Z^n$.

Note that both $V$ and $V_0$ are $K$-Banach spaces with the norm induced from the $p$-adic valuation: $$v_p\left(\sum_{\alpha\in \Z^n}\lambda_{\alpha}X_1^{\alpha_1}\cdots X_n^{\alpha_n} \right)\coloneqq \inf_{\alpha \in \Z^n}v(\lambda_{\alpha})$$
and, from the definition of $\rho$, we have $g\cdot B_V \subseteq B_V,$ for any $g\in \g$ so $\left(G-1 \right) \cdot B_V \subseteq p B_V$ since $e^{pg}-1\in p\,\wh{U(\g_R)}$.

It follows from \cite[Theorem 7.3]{ardakov2013irreducible} that $\Weyl$ is a simple ring, so $V_0$ and $V$ are faithful $\Weyl$-modules. This can also be seen explicitly as follows, where $$\Ann_{\Weyl}(S) \coloneqq\left\{\zeta\in\Weyl: \zeta\cdot s=0 \text{ for all }s\in S\right\}$$ for a subset $S\subseteq V_0$.

\begin{lemma}\label{Lmm: V is faithful}
    For $K[X_1, \ldots, X_n]\subseteq V_0$ we have $\Ann_{\Weyl} K[X_1, \ldots, X_n]=0$. In particular, $V_0$ and $V$ are faithful $\Weyl$-modules.
\end{lemma}
\begin{proof}
    Take $$\zeta=\sum_{\alpha,\beta\in\N_0^n}\lambda_{\alpha,\beta}x_1^{\alpha_1}\cdots x_n^{\alpha_n}\del_1^{\beta_1}\cdots\del_n^{\beta_n}\in \Ann_{\Weyl} K[X_1, \ldots, X_n].$$
    We consider the lexicographic order $<$ on $\N_0^n$ and argue by induction on $\beta\in\N_0^n$ that $\lambda_{\alpha,\beta}=0$ for all $\alpha\in \N_0^n$. Indeed, for $\beta=0$ we have $$0=\zeta\cdot 1=\sum_{\alpha\in \N_0^n}\lambda_{\alpha,0} X_1^{\alpha_1}\cdots X_n^{\alpha_n},$$ so $\lambda_{\alpha,0}=0$ for all $\alpha \in \N_0^n$.
    More generally, note that for $\beta, \gamma\in\N_0^n$ with $\gamma< \beta$ we have $$\del_1^{\beta_1}\cdots\del_n^{\beta_n}\cdot X_1^{\gamma_1}\cdots X_n^{\gamma_n}=0.$$
    By induction, fix $\gamma\in\N_0^n$ such that for any $\beta\in\N_0^n$ with $\beta<\gamma$ we have $\lambda_{\alpha,\beta}=0$. Then
    $$\zeta \cdot X_1^{\gamma_1}\cdots X_n^{\gamma_n}=\gamma_1!\cdots\gamma_n!\sum_{\alpha \in \N_0^n}\lambda_{\alpha, \gamma}X_1^{\alpha_1}\cdots X_n^{\alpha_n},$$ so $\lambda_{\alpha,\gamma}=0$ for all $\alpha \in \N_0^n$. The conclusion then follows.
\end{proof}
In particular, whenever we have a $K$-algebra $S$ with an algebra homomorphism $\varphi: S \to \Weyl$, we have that $V_0$ is a faithful $S$-module along $\varphi$ if and only if $\varphi$ is injective.

\subsection{The Multiplication Map} \label{Subsec: Multiplication map}

In this subsection, we show that the multiplication map $KC \otimes_K \rho \left( \hyp{c} \right) \to \rho \left( \aff{c} \right)$ is injective. We do this by induction on the following subalgebras, noting $\c^n = \c$:
\begin{align*}
    \c^k &\coloneqq \langle c_{ij}: 1\le i,j\le k\rangle_{\Z_p} \oplus \langle c_{1j}: k+1\leq j \leq  n\rangle_{\Z_p}\\  \tilde{\c}^k &\coloneqq  \langle c_{ik} : 1 \leq i \leq k \rangle_{\Z_p} \oplus \langle c_{1j} : k+1 \leq j \leq n \rangle_{\Z_p} .
\end{align*}
\cref{Lmm: Zero annihilator} shows this for the $k=1$ case, by a general computation. We then find it useful to introduce the following subalgebras. Generators of the images of these subalgebras under $\rho$, as well as of $\rho(\c^k)$ and $\rho(\tilde{\c}^k)$, are found in \cref{Lmm: Automorphic images} and are shown below for $2\le k\le n$.
\begin{align*}
    \c_+^k & \coloneqq \sigma_k\cdots\sigma_n\left(\c^k\right), & 
        \c_-^k & \coloneqq \sigma_1\cdots\sigma_{k-1}\left(\c^k\right) \\
    \tilde{\c}_+^k &\coloneqq  \sigma_k \cdots \sigma_n \left( \tilde{\c}^k \right), &
        \tilde{\c}_-^k &\coloneqq \sigma_1\cdots\sigma_{k-1}\left(\tilde{\c}^k\right)
\end{align*}
\begin{figure}[!htb]
\setlength{\tabcolsep}{3pt}
\centering
    \begin{subfigure}{0.45\linewidth}
    \centering
    \begin{tabular}{cccccc}
         $x_1^2$&  $\cdots$&  $x_1x_{k-1}$&  $x_1x_k$&  $\cdots$& $x_1x_n$\\
         &  $\ddots$&  $\vdots$&  $\vdots$&  & \\
         &  &  $x_{k-1}^2$&  $x_{k-1}x_k$&  & \\[0.5em]
         &  &  &  $x_k^2$&  & \\
    \end{tabular}
    \caption{Generators of $\rho(\c^k)$}
    \end{subfigure}
    \vspace{1em}
    \begin{subfigure}{0.45\linewidth}
    \centering
    \begin{tabular}{ccc}
               $x_1x_k$&  $\cdots$& $x_1x_n$\\
               $\vdots$&  & \\
               $x_{k-1}x_k$&  & \\[0.5em]
               $x_k^2$&  & \\
    \end{tabular}
    \caption{Generators of $\rho(\tilde{\c}^k)$}
    \end{subfigure}
    \\[1ex]
    \begin{subfigure}{0.45\linewidth}
    \centering
    \begin{tabular}{cccccc}
         $x_1^2$&  $\cdots$&  $x_1x_{k-1}$&  $x_1\del_k$&  $\cdots$& $x_1\del_n$\\
         &  $\ddots$&  $\vdots$&  $\vdots$&  & \\
         &  &  $x_{k-1}^2$&  $x_{k-1}\del_k$&  & \\[0.5em]
         &  &  &  $\del_k^2$&  & \\
    \end{tabular}
    \caption{Generators of $\rho(\c_+^k)$}
    \end{subfigure}
    \vspace{1em}
    \begin{subfigure}{0.45\linewidth}
    \centering
    \begin{tabular}{ccc}
               $x_1\del_k$&  $\cdots$& $x_1\del_n$\\
               $\vdots$&  & \\
               $x_{k-1}\del_k$&  & \\[0.5em]
               $\del_k^2$&  & \\
    \end{tabular}
    \caption{Generators of $\rho(\tilde{\c}_+^k)$}
    \end{subfigure}
    \\[1ex]
    \begin{subfigure}{0.45\linewidth}
    \centering
    \begin{tabular}{cccccc}
         $\del_1^2$&  $\cdots$&  $\del_1\del_{k-1}$&  $\del_1x_k$&  $\cdots$& $\del_1x_n$\\
         &  $\ddots$&  $\vdots$&  $\vdots$&  & \\
         &  &  $\del_{k-1}^2$&  $\del_{k-1}x_k$&  & \\[0.5em]
         &  &  &  $x_k^2$&  & \\
    \end{tabular}
    \caption{Generators of $\rho(\c_-^k)$}
    \end{subfigure}
    \begin{subfigure}{0.45\linewidth}
    \centering
    \begin{tabular}{ccc}
               $\del_1x_k$&  $\cdots$& $\del_1x_n$\\
               $\vdots$&  & \\
               $\del_{k-1}x_k$&  & \\[0.5em]
               $x_k^2$&  &
    \end{tabular}
    \caption{Generators of $\rho(\tilde{\c}_-^k)$}
    \end{subfigure}
\end{figure}

Note that $\c^k_+$ is obtained from $\c^k$ by a Fourier transform that converts the basis elements in $\c^k \setminus \c^{k-1}$ into elements of $\a \oplus \b$. This allows us to apply \cref{Prop: Gluing II} with $\n = \c^{k-1}$ and $\h = \tilde{\c}_+^k$ to obtain the result. In order to verify the faithfulness condition for $KH \otimes_K T_K$ in \cref{Lmm: Induction base case}, we induct over $\tilde{\c}^k_-$ which is obtained by a total Fourier transform of $\tilde{\c}^k_+$. Specifically, we apply \cref{Prop: Gluing II} with $\n = \left\langle c_{ik} : 1 \leq i \leq n \right\rangle_{\Z_p}$, and $\h = \a \cap \c^k_-$. This concludes the proof.

We now begin by noting that the image $\rho\left(\aff{c}\right)$ embeds into $V$ by the map $x_i \mapsto X_i$, for $1 \leq i \leq n$, and under this identification, the action of $\aff{c}$ on $V$ is given by multiplication $x \cdot v = \rho(x) v$. Then since $V$ is a domain, for any non-zero $v\in V$ and subalgebra $\c' \subseteq \c$ we have that 
\begin{align*}
    \Ann_{\hyp{c'}}(v)&=\ker\restr{\rho}{\hyp{c'}}, & \Ann_{\aff{c'}}(v)&=\ker\restr{\rho}{\aff{c'}},
\end{align*}
and so in particular they are independent of the non-zero $v\in V$. There is a special class of subalgebras for which these annihilators are zero.

\begin{lemma}
\label{Lmm: Zero annihilator}
    Let $I\subseteq\{1, \ldots, n\}^2$ such that at most one of $(i,j)$ and $(j,i)$ is in $I$ for each $i,j\in\{1, \ldots, n\}$, and suppose $f=(f_1, \ldots, f_n):\N_0^I\to\N_0^n$ given by $$f_k((\alpha_{ij})_{(i, j)\in I})=\sum_{(i, j)\in I}\alpha_{ij}(\delta_{ik}+\delta_{kj})$$ is injective. Then $\ker\restr{\rho}{\aff{c'}}=0$ for $\c'=\left\langle c_{ij}: (i, j)\in I\right\rangle_{\Z_p}$ and so $$KC'\otimes_K
    \rho\left(\hyp{c'} \right)\to\rho\left(\aff{c'}\right)$$ is injective.
\end{lemma}

\begin{proof}
    Take $$\zeta=\sum_{\alpha\in\N_0^I} \lambda_{\alpha} \left(\prod_{(i, j)\in I}c_{ij}^{\alpha_{ij}} \right)\in\ker\restr{\rho}{\aff{c'}}$$ for some $\lambda_{\alpha}\in K$ uniformly bounded and note that by construction $$0=\rho(\zeta)=\sum_{\alpha\in\N_0^I}(-1)^{|\alpha|}\lambda_{\alpha} \mathbf{x}^{f(\alpha)}.$$ Since $f$ is injective, then $\lambda_{\alpha}=0$ for all $\alpha\in\N_0^I$ and the first part follows. Finally, by \cite[Theorem 3.2]{ardakov2013verma} we have that $KC'\otimes_K\hyp{c'}\to \aff{c'}$ is injective so the last part also follows.
\end{proof}

Recall the following subalgebras for $1\le k\le n$, introduced at the beginning of the section. 
\begin{align*}
        \c^k&\coloneqq \langle c_{ij}: 1\le i,j\le k\rangle_{\Z_p} \oplus \langle c_{1j}: k+1\leq j \leq  n\rangle_{\Z_p}, & 
        \c_+^k& \coloneqq \sigma_k\cdots\sigma_n\left(\c^k\right), &
        \c_-^k&\coloneqq \sigma_1\cdots\sigma_{k-1}\left(\c^k\right)\\
        \tilde{\c}^k&\coloneqq  \langle c_{ik} : 1 \leq i \leq k \rangle_{\Z_p} \oplus \langle c_{1j} : k+1 \leq j \leq n \rangle_{\Z_p}, &
        \tilde{\c}_+^k&\coloneqq  \sigma_k \cdots \sigma_n \left( \tilde{\c}^k \right), &
        \tilde{\c}_-^k&\coloneqq \sigma_1\cdots\sigma_{k-1}\left(\tilde{\c}^k\right)
    \end{align*} 
    Note that these are all abelian by \cref{Lmm: Commutators}, and also that $\tilde{\c}^k\subseteq\c^k$, $\tilde{\c}_+^k\subseteq\c_+^k$, $\tilde{\c}_-^k\subseteq\c_-^k$ and $\c^n = \c$. 

\begin{lemma}\label{Lmm: Automorphic images}
    We have 
    \begin{align*}
        \c_+^k&=\begin{cases}\langle b_{1i}: 1 \leq i \leq n \rangle_{\Z_p} & \text{ if }k=1,\\ \langle c_{ij}: 1\le i, j\le k-1\rangle_{\Z_p} \oplus\langle a_{ki} : 1 \leq i \leq k-1 \rangle_{\Z_p} \oplus \langle a_{i1} : k+1 \leq i \leq n \rangle_{\Z_p} \oplus \langle b_{kk}\rangle_{\Z_p} & \text{ if } 2\le k\le n;\end{cases}\\
        \c_-^k&=\begin{cases}\langle c_{1i}: 1 \leq i \leq n \rangle_{\Z_p} & \text{ if }k=1,\\ \langle b_{ij}: 1\le i, j\le k-1\rangle_{\Z_p} \oplus\langle a_{ik} : 1 \leq i \leq k-1 \rangle_{\Z_p} \oplus \langle a_{1i} : k+1 \leq i \leq n \rangle_{\Z_p} \oplus \langle c_{kk}\rangle_{\Z_p} & \text{ if } 2\le k\le n;\end{cases}\\
        \tilde{\c}_+^k&=\begin{cases}\langle b_{1i}: 1 \leq i \leq n \rangle_{\Z_p} & \text{ if }k=1,\\ \langle a_{ki} : 1 \leq i \leq k-1 \rangle_{\Z_p} \oplus \langle a_{i1} : k+1 \leq i \leq n \rangle_{\Z_p} \oplus \langle b_{kk}\rangle_{\Z_p} & \text{ if } 2\le k\le n;\end{cases}\\
        \tilde{\c}_-^k&=\begin{cases}\langle c_{1i}: 1 \leq i \leq n \rangle_{\Z_p} & \text{ if }k=1,\\ \langle a_{ik} : 1 \leq i \leq k-1 \rangle_{\Z_p} \oplus \langle a_{1i} : k+1 \leq i \leq n \rangle_{\Z_p} \oplus \langle c_{kk}\rangle_{\Z_p} & \text{ if } 2\le k\le n.\end{cases}\\
    \end{align*} 
\end{lemma}
\begin{proof}
    For $k=1$ by definition we have $$\c^1=\tilde{\c}^1=\c_-^1=\tilde{\c}_-^1=\langle c_{1i}: 1\le i\le n\rangle_{\Z_p}$$ and by \cref{Lmm: Weyl automorphisms} we have \begin{align*}
        c_{11}&\xmapsto{\sigma_n}\cdots\xmapsto{\sigma_2} c_{11}\xmapsto{\sigma_1} -b_{11}, & \\ c_{1i}&\xmapsto{\sigma_n}\cdots\xmapsto{\sigma_{i+1}}c_{1i}\xmapsto{\sigma_i}a_{i1}\xmapsto{\sigma_{i-1}}\cdots\xmapsto{\sigma_2}a_{i1}\xmapsto{\sigma_1} -b_{1i} &\text{ if } 2\le i\le n,
    \end{align*}
    giving the desired equalities for $\c_+^1$ and $\tilde{\c}_+^1.$
    
    Now for $2\le k\le n$, by \cref{Lmm: Weyl automorphisms} we have \begin{align*}
        c_{ij}&\xmapsto{\sigma_n}\cdots\xmapsto{\sigma_k} c_{ij} &\text{ if } 1\le i, j\le k-1,\\
        c_{ki}&\xmapsto{\sigma_n}\cdots\xmapsto{\sigma_{k-1}} c_{ki}\xmapsto{\sigma_k} a_{ki} &\text{ if } 1\le i\le k-1 \\
        c_{kk}&\xmapsto{\sigma_n}\cdots\xmapsto{\sigma_{k-1}} c_{kk}\xmapsto{\sigma_k} -b_{kk} & \\      c_{1i}&\xmapsto{\sigma_n}\cdots\xmapsto{\sigma_{i+1}}c_{1i}\xmapsto{\sigma_i}a_{i1}\xmapsto{\sigma_{i-1}}\cdots\xmapsto{\sigma_k} a_{i1} &\text{ if } k+1\le i\le n
    \end{align*}
    giving the desired equalities for $\c_+^k$ and $\tilde{\c}_+^k$. Moreover
    \begin{align*}
        c_{ij}&\xmapsto{\sigma_{k-1}}\cdots\xmapsto{\sigma_{j+1}}c_{ij}\xmapsto{\sigma_j}a_{ji}\xmapsto{\sigma_{j-1}}\cdots\xmapsto{\sigma_{i+1}}a_{ji}\xmapsto{\sigma_i}-b_{ij}\xmapsto{\sigma_{i-1}}\cdots\xmapsto{\sigma_1}-b_{ij} &\text{ if } 1\le i< j\le k-1,\\
        c_{ii}&\xmapsto{\sigma_{k-1}}\cdots\xmapsto{\sigma_{i+1}}c_{ii}\xmapsto{\sigma_i}-b_{ii}\xmapsto{\sigma_{i-1}}\cdots\xmapsto{\sigma_1}-b_{ii} &\text{ if } 1\le i\le k-1,\\
        c_{ki}&\xmapsto{\sigma_{k-1}}\cdots\xmapsto{\sigma_{i+1}}c_{ki}\xmapsto{\sigma_i}a_{ik}\xmapsto{\sigma_{i-1}}\cdots\xmapsto{\sigma_1} a_{ik} &\text{ if } 1\le i\le k-1 \\
        c_{kk}&\xmapsto{\sigma_{k-1}}\cdots\xmapsto{\sigma_1} c_{kk} & \\      c_{1i}&\xmapsto{\sigma_{k-1}}\cdots\xmapsto{\sigma_2}c_{1i}\xmapsto{\sigma_1}a_{1i} &\text{ if } k+1\le i\le n
    \end{align*}
    giving the equalities for $\c_-^k$ and $\tilde{\c}_-^k$.
\end{proof}

\begin{lemma}\label{Lmm: Induction base case}
    For $2\le k\le n$, the map $K\tilde{C}_+^k\otimes_K \rho\left(U\left(\c_+^k\right)_K\right)\to \Weyl$ is injective.
\end{lemma} 
\begin{proof}
    Fix $2\le k\le n$ and let 
    \begin{align*}
        \a_+^k&\coloneqq\a\cap\c_+^k=\a\cap\tilde{\c}_+^k=\langle a_{ki}: 1\le i\le k-1\rangle_{\Z_p} \oplus\langle a_{i1}: k+1\le i\le n \rangle_{\Z_p},\\
        \a_-^k&\coloneqq\a\cap\c_-^k=\a\cap\tilde{\c}_-^k=\langle a_{ik}: 1\le i\le k-1\rangle_{\Z_p} \oplus\langle a_{1i}: k+1\le i\le n \rangle_{\Z_p},
    \end{align*}
    
    where the equalities follow by \cref{Lmm: Automorphic images}.
    We proceed in three steps.
    
    \noindent\textbf{Step 1:} $\Ann_{\widehat{U\left(\a_+^k\right)}_K} K[X_1, \ldots, X_n]=0$. 
    
    First note the action is well-defined since $\a_+^k$ acts by homogeneous operators of degree zero. Now take $$\xi=\sum_{\alpha, \beta} \mu_{\alpha, \beta}a_{k1}^{\alpha_1}\ldots a_{k, k-1}^{\alpha_{k-1}}a_{k+1, 1}^{\beta_{k+1}}\ldots a_{n1}^{\beta_n}\in \Ann_{\widehat{U(\a_+^k)}_K} K[X_1, \ldots, X_n]$$ 
    where the sum is over $\alpha=(\alpha_1, \ldots, \alpha_{k-1})\in\N_0^{k-1}$ and $\beta=(\beta_{k+1}, \ldots, \beta_n)\in\N_0^{n-k}$ and $\mu_{\alpha, \beta}\in K$ satisfy $\mu_{\alpha, \beta}\to 0$ as $|\alpha|+|\beta|\to\infty$; in particular, they are uniformly bounded. Recall that $\rho(a_{ij})=-x_j\del_i$ so for any $\gamma=(\gamma_1, \ldots, \gamma_n)\in\N_0^n$ we have 
    $$0=\xi\cdot X_1^{\gamma_1}\cdots X_n^{\gamma_n}=\sum_{\alpha, \beta}c_{\alpha, \beta}^{(\gamma)}X_1^{\gamma_1+\alpha_1+|\beta|}X_2^{\gamma_2 + \alpha_2}\cdots X_{k-1}^{\gamma_{k-1}+\alpha_{k-1}}X_k^{\gamma_k-|\alpha|}X_{k+1}^{\gamma_{k+1}-\beta_{k+1}}\cdots X_n^{\gamma_n-\beta_n},$$
    where $$c_{\alpha, \beta}^{(\gamma)}=(-1)^{|\alpha|+|\beta|}\mu_{\alpha, \beta}\frac{\gamma_k!\cdots\gamma_n!}{(\gamma_k-|\alpha|)!(\gamma_{k+1}-\beta_{k+1})!\cdots (\gamma_n-\beta_n)!}$$
    whenever $|\alpha|\le \gamma_k$ and $\beta_j\le \gamma_j$ for $k+1\le j\le n$, and $c_{\alpha, \beta}^{(\gamma)}=0$ otherwise. Since 
    $$(\alpha, \beta)\in\N_0^{n-1}\mapsto (\gamma_1+\alpha_1+|\beta|, \alpha_2+\gamma_2, \ldots, \alpha_{k-1}+\gamma_{k-1}, \gamma_k-|\alpha|, \gamma_{k+1}-\beta_{k+1}, \ldots, \gamma_n-\beta_n)\in\N_0^n$$ is injective for any $\gamma\in\N_0^n$ we get that $c_{\alpha, \beta}^{(\gamma)}=0$ for all $\alpha\in\N_0^{k-1}, \beta\in\N_0^{n-k}$ and $\gamma\in\N_0^n$. Then $\mu_{\alpha, \beta}=0$ whenever $|\alpha|\le \gamma_k$ and $\beta_j\le \gamma_j$ for all $k+1\le j\le n$. Since $\gamma$ is arbitrary we get $\xi=0$. 

    \noindent\textbf{Step 2:} $K[X_1, \ldots, X_n]$ is faithful as a $KA_+^k\otimes_K K[x_1, \ldots, x_{k-1}, \del_k, \ldots, \del_n]$-module. 
    
    Take $$\zeta=\sum_{\alpha, \beta}\zeta_{\alpha, \beta} x_1^{\alpha_1}\cdots x_{k-1}^{\alpha_{k-1}}\del_k^{\beta_k}\cdots\del_n^{\beta_n}\in\Ann_{KA_+^k\otimes_K K[x_1, \ldots, x_{k-1}, \del_k, \ldots, \del_n]}K[X_1, \ldots, X_n]$$
    where the sum is over $\alpha=(\alpha_1, \ldots, \alpha_{k-1})\in\N_0^{k-1}$ and $\beta=(\beta_k, \ldots, \beta_n)\in\N_0^{n-k+1}$, with only finitely many $\zeta_{\alpha, \beta}\in KA_+^k$ non-zero. In particular, there are $d_1, d_2\in\N_0$ such that $\zeta_{\alpha, \beta}=0$ whenever $|\alpha|\ge d_1$ or $|\beta|\ge d_2$.
    Now for any monomial $f\in K[X_1, \ldots, X_n]$ we have
    $$0=X_1^{d_1+d_2}\del_k^{d_1}\zeta\cdot f=\sum_{\alpha, \beta} (-1)^{d_1+\abs{\beta}} \zeta_{\alpha, \beta}a_{k1}^{\beta_k+d_1-(\alpha_2+\cdots+\alpha_{k-1})}a_{k2}^{\alpha_2}\cdots a_{k, k-1}^{\alpha_{k-1}} a_{k+1, 1} ^{\beta_{k+1}}\cdots a_{n, 1}^{\beta_n} \cdot X_1^{d_2+|\alpha|-|\beta|}f.$$ 
    Fix $d \geq -d_2$. Since $\a_+^k$ acts by homogeneous operators of degree $0$, looking at the terms of total degree $d+d_2+\deg(f)$ in $X_1^{d_1+d_2}\partial_k^{d_1} \zeta \cdot f$ we see that 
    $$0=X_1^{d_2+d}\left(\sum_{|\alpha|-|\beta|=d}(-1)^{d_1+\abs{\beta}} \zeta_{\alpha, \beta}a_{k1}^{\beta_k+d_1-(\alpha_2+\cdots+\alpha_{k-1})}a_{k2}^{\alpha_2}\cdots a_{k, k-1}^{\alpha_{k-1}} a_{k+1, 1} ^{\beta_{k+1}}\cdots a_{n, 1}^{\beta_n}\cdot f \right).$$
    But since $K[X_1, \ldots, X_n]$ is a domain and $f$ is an arbitrary monomial we have that 
    $$\sum_{|\alpha|-|\beta|=d}(-1)^{d_1+\abs{\beta}} \zeta_{\alpha, \beta}a_{k1}^{\beta_k+d_1-(\alpha_2+\cdots+\alpha_{k-1})}a_{k2}^{\alpha_2}\cdots a_{k, k-1}^{\alpha_{k-1}} a_{k+1, 1} ^{\beta_{k+1}}\cdots a_{n, 1}^{\beta_n}\in\Ann_{\widehat{U(\a_+^k)}_K} K[X_1, \ldots, X_n]=0.$$ Finally by \cite[Theorem 3.2]{ardakov2013verma} we get $\zeta_{\alpha, \beta}=0$ whenever $|\alpha|-|\beta|=d$. Since $d$ is arbitrary we get $\zeta=0$.

    \noindent\textbf{Step 3:} $V_0$ is faithful as a $K\tilde{C}_+^k\otimes_K \rho(U(\c_+^k)_K)$-module.

    By applying the automorphism $\sigma\otimes_K\tau$ it is enough to show that $V_0$ is faithful as a $K\tilde{C}_-^k\otimes_K \rho(U(\c_-^k)_K)$-module.
    
    We apply \cref{Prop: Gluing II} with
    \begin{align*}
        \n &= \langle c_{ik}: 1\le i\le n\rangle_{\Z_p}, & \h &=
        \a_-^k, \\
        T & = \rho\left(U\left(\c_-^k\otimes_{\Z_p} R\right)\right), & \V &= \left\{X_k^{-2nr}: r\in \N_0\right\} \subseteq V,
    \end{align*}
     where we filter $T\subseteq A_n(R)$ by total degree. Also note that $\n \oplus \h $ is a subalgebra since by \cref{Lmm: Commutators} we have \begin{align*}
         [c_{ik}, a_{jk}]&= \delta_{ij}c_{kk} \text{ for } 1\le i\le n, \, 1\le j\le k-1 \\
         [c_{ik}, a_{1j}]&= \delta_{i1}c_{jk} \text{ for } 1\le i\le n, \, k+1\le j\le n
     \end{align*} 
    and that $\tilde{\c}_-^k\subseteq \n \oplus \h$ by \cref{Lmm: Automorphic images}. For $v_r=X_k^{-2nr}\in \V$ let
    \begin{equation*}
        W_{v_r} \coloneqq  \langle X_1^{\alpha_1} \cdots X_n^{\alpha_n} : \alpha \in \N_0^n,\, \alpha_1, \ldots, \alpha_{k-1}\leq r,\, \alpha_1+\alpha_i\le r\text{ for } k+1\le i\le n, \, |\alpha| \text{ even} \rangle_K \subseteq \hyp{\n} \cdot X_k^{-2nr}
    \end{equation*}
    where the inclusion follows since 
    $$X_1^{\alpha_1}\cdots X_n^{\alpha_n}= (-1)^{nr+\frac{\abs{\alpha}}{2}}
        c_{1k}^{\alpha_1}\cdots c_{k-1, k}^{\alpha_{k-1}}c_{k k}^{\alpha_k+nr-\frac{\abs{\alpha}}{2}}c_{k+1, k}^{\alpha_{k+1}}\cdots c_{nk}^{\alpha_n}\cdot X_k^{-2nr}$$ for $\alpha=(\alpha_1, \ldots,\alpha_n)\in\N_0^n$ with $|\alpha|$ even and $\alpha_1, \ldots, \alpha_{k-1}\leq r$, $\alpha_1+\alpha_i\le r$ for $k+1\le i\le n$, since in particular $|\alpha|\le 2nr$.
    Each $W_{v_r}$ is both a $\aff{h}$-submodule since it is stable by the actions of $x_k\del_1, \ldots, x_k\del_{k-1}$ and $x_{k+1}\del_1, \ldots, x_n\del_1$, and a $T_K$-submodule as it is also stable under the action of $\del_i\del_j$ for $1\le i, j\le k-1$ and under the action of $x_k^2$. Then $W = \langle X_1^{\alpha_1}\cdots X_n^{\alpha_n}: |\alpha|\text{ even}\rangle_K$ and we have
    \begin{itemize}
        \item For $f\in W$ we have $$RH \otimes_R F_dT \cdot f \subseteq \left\{ g \in K[X_1,\ldots,X_n] : \deg g \leq d + \deg f,\, v_p(g) \geq v_p(f) \right\},$$
        which is finitely generated as an $R$-module, hence so is $RH \otimes_R F_dT \cdot f$ as $R$ is noetherian.
        \item $KH\otimes_K T_K$ acts faithfully on $W$ since if $\zeta\in KH\otimes_K T_K$ annihilates $W$, it also annihilates $K[X_1, \ldots, X_n]=W+\del_1\cdot W$ since $\del_1$ commutes with the image of $KH\otimes_K T_K$ in $\Weyl$, and then $\zeta=0$ by Step 2, after applying the automorphism $\sigma\otimes_K\tau$.
        \item For every $v\in\V$, the multiplication map  $KN \otimes_K \rho \left( \hyp{n} \right) \to \rho \left( \aff{n} \right)$ is injective by \cref{Lmm: Zero annihilator} as $$(\alpha_{1k}, \ldots, \alpha_{nk})\in\N_0^n\mapsto (\alpha_{1k}, \ldots, \alpha_{k-1, k}, \alpha_{kk}+|\alpha|, \alpha_{k+1, k}, \ldots, \alpha_{nk})\in\N_0^n$$ is injective.
    \end{itemize}
    The result now follows from \cref{Prop: Gluing II} 
\end{proof}

\begin{proposition}\label{Prop: Multiplication map}
    The multiplication map $KC \otimes_K \rho \left( \hyp{c} \right) \to \rho \left( \aff{c} \right)$ is injective.
\end{proposition}

\begin{proof}
    We show by induction that the multiplication map $KC^k \otimes_K \rho \left( \hyp{c^{\textit{k}}} \right) \to \rho \left( \aff{c^{\textit{k}}} \right)$ is injective for $1 \leq k \leq n$. The base case $k=1$ follows from \cref{Lmm: Zero annihilator} as the map 
    $$(\alpha_{11}, \ldots, \alpha_{1n})\in\N_0^n\mapsto (2\alpha_{11}+\alpha_{12}+\cdots+\alpha_{1n}, \alpha_{12}, \ldots, \alpha_{1n})\in\N_0^n$$ is injective.
    Now, for the induction, fix $2\le k\le n$ and note that it is enough to show that $V_0$ is faithful over the algebra $KC^k \otimes_K \rho \left( U(\c^k )_K \right)$. Applying the automorphism $\sigma_k\cdots\sigma_n$ this is equivalent to showing that $V_0$ is faithful over the algebra $KC_+^k\otimes_K \rho \left( U(\c_+^k)_K \right)$.
    The following pictures provide an illustration of this argument for $\sp_4$. 
    
\begin{figure}[!htb]
    \centering
    \begin{subfigure}{.49\textwidth}
        \centering
        \includegraphics[page=5,valign=c]{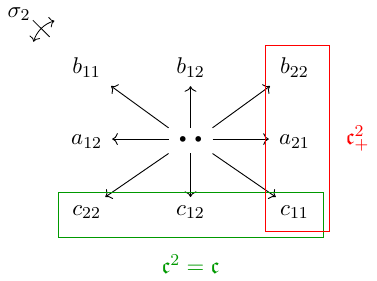}
        \caption{Action of $\sigma_2$}
    \end{subfigure}
    \begin{subfigure}{.49\textwidth}
        \centering
        \includegraphics[page=6,valign=c]{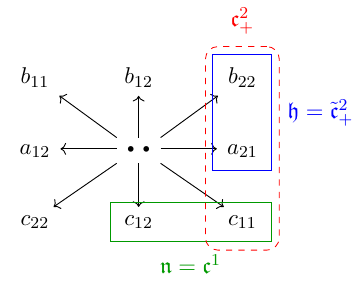}
        \caption{Applying \cref{Prop: Gluing II}}
    \end{subfigure}
\end{figure}

    We then apply \cref{Prop: Gluing II} with  
    \begin{align*}
        \n &= \c^{k-1}, & \h &= \tilde{\c}_+^k, \\
        T & = \rho \left(U(\c_+^k\otimes_{\Z_p} R) \right)\subseteq R[x_1, \ldots, x_{k-1}, \del_k, \ldots, \del_n] , & \V &= \left\{ X_1^{-2nr} : r \in \N_0 \right\} \subseteq V,
    \end{align*}
    where we filter $T$ by total degree, $\n \oplus \h $ is a subalgebra by \cref{Lmm: Commutators} and $\c_+^k \subseteq \n \oplus \h$ by \cref{Lmm: Automorphic images}. For $v_r=X_1^{-2nr} \in \V$ let
    \begin{equation*}
        W_{v_r} \coloneqq  \langle X_1^{\alpha_1} \cdots X_n^{\alpha_n} : \alpha \in \N_0^n,\, \alpha_k, \ldots, \alpha_n \leq r,\, \abs{\alpha} \text{ is even} \rangle_K \subseteq \hyp{\n} \cdot X_1^{-2nr}
    \end{equation*}
    where the inclusion follows since
    $$X_1^{\alpha_1}\cdots X_n^{\alpha_n}= (-1)^{nr + \abs{\alpha}/2}c_{11}^{\frac{\alpha_1}{2}+nr-\left\{\frac{\alpha_2}{2}\right\}-\cdots-\left\{\frac{\alpha_{k-1}}{2}\right\}-\frac{\alpha_k+\cdots+\alpha_n}{2}} \bigg(c_{22}^{\lfloor\frac{\alpha_2}{2}\rfloor}\cdots c_{k-1,k-1}^{\lfloor\frac{\alpha_{k-1}}{2}\rfloor}\bigg)\bigg(c_{12}^{2\left\{\frac{\alpha_2}{2}\right\}}\cdots c_{1, k-1}^{2\left\{\frac{\alpha_{k-1}}{2}\right\}}\bigg)\left(c_{1k}^{\alpha_k} \cdots c_{1n}^{\alpha_n}\right) \cdot X_1^{-2nr}$$
    for $\alpha=(\alpha_1, \ldots, \alpha_n)\in\N_0^n$ with $\alpha_i\le r$ for all $i\ge k$ and $\abs{\alpha}$ even. Here $\{\cdot\}$ denotes the fractional part.

    These are both a $\aff{h}$-submodule and a $T_K$-submodule as both $\h = \tilde{\c}_+^k$ and $\c_+^k$ act by homogeneous operators of even degree, which can only increase the degrees of $X_1,\ldots,X_{k-1}$. Then $W = \langle X_1^{\alpha_1} \cdots X_n^{\alpha_n} :  \abs{\alpha} \text{ is even} \rangle_K$ and we have:
    \begin{itemize}
        \item For $f\in W$ we have $$RH \otimes_R F_dT \cdot f \subseteq \left\{ g \in K[X_1,\ldots,X_n] : \deg g \leq d + \deg f,\, v_p(g) \geq v_p(f) \right\},$$
        which is finitely generated as an $R$-module, hence so is $RH \otimes_R F_dT \cdot f$ as $R$ is noetherian.
        \item $KH\otimes_K T_K$ acts faithfully on $W$ by \cref{Lmm: Induction base case}.
        \item For every $v\in\V$, the multiplication map is injective by the inductive hypothesis.
    \end{itemize}
    The result now follows from \cref{Prop: Gluing II}.
\end{proof}

\subsection{The subalgebra \texorpdfstring{$KA$}{KA}}\label{Subsec: KA}
The next step is to show that $KA$ acts faithfully on $V_0$, so that $\restr{\rho}{KA}$ is injective. We start by looking at a Cartan subalgebra $\a_0 \coloneqq  \langle a_{ii} : 1 \leq i \leq n\rangle_{\Z_p}$ of $\g$.

\begin{lemma}\label{Prop: Cartan faithful}
    $KA_0$ acts faithfully on $K[X_1,\dots, X_n] \subseteq V_0$.
\end{lemma}
\begin{proof}
    We actually show that $\aff[0]{a}$ acts faithfully on $K[X_1,\dots,X_n]$, and note that the action is well-defined as $\aff[0]{a}$ acts on $V_0$ by homogeneous operators of degree 0. Let $$\zeta=\sum_{\alpha\in\N_0^n}\lambda_{\alpha}a_{11}^{\alpha_1} \cdots a_{n n}^{\alpha_n} \in \Ann_{\aff[0]{a}} K[X_1,\dots,X_n]$$ for $\lambda_{\alpha}\in K$ with $\lambda_{\alpha}\to 0$ as $|\alpha|\to\infty$.  
    Then for each $\beta \in \N_0^n$, since $\rho(a_{ii})=-\frac{1}{2}-x_i\del_i$ for $1\le i\le n$ we have
    $$0 = \zeta \cdot X_1^{\beta_1}\cdots X_n^{\beta_n} = \sum_{\alpha\in\N_0^n}\lambda_{\alpha} \left(\prod_{i=1}^n \left(-\frac{1}{2}-\beta_i\right)^{\alpha_i} \right) 
    X_1^{\beta_1}\cdots X_n^{\beta_n}=f\left(-\frac{1}{2}-\beta_1, \ldots, -\frac{1}{2}-\beta_n\right)X_1^{\beta_1}\cdots X_n^{\beta_n}$$ where $$f(X_1, \ldots, X_n)=\sum_{\alpha\in\N_0^n}\lambda_{\alpha}X_1^{\alpha_1}\cdots X_n^{\alpha_n}\in K\langle X_1, \ldots, X_n\rangle.$$ Then by \cite[Lemma 4.7]{ardakov2013verma} it follows that $\lambda_{\alpha}=0$ for all $\alpha\in \N_0^n$, and hence $\zeta = 0$. 
\end{proof}
\begin{proposition}\label{Prop: KA faithful}
    $\restr{\rho}{KA}$ is injective.
\end{proposition}
\begin{proof}
    For $1 \leq k \leq n$ let 
    \begin{equation*}
        \a_k\coloneqq \langle a_{ij}:1\le i, j\le k \rangle_{\Z_p}\oplus\langle a_{ii}: k+1\le i\le n\rangle_{\Z_p} \qquad
        \b_k\coloneqq \langle b_{ik}: 1\le i\le k-1\rangle_{\Z_p} \qquad 
        \c_k\coloneqq \langle c_{ik}: 1\le i\le k-1\rangle_{\Z_p}
    \end{equation*}
    \begin{equation*}
        \a_k^-\coloneqq \a_{k-1}\oplus\c_k \qquad
        \a_k^+\coloneqq \sigma(\a_k^-)=\a_{k-1}\oplus\b_k \qquad
        \a_k^{\pm}\coloneqq \sigma_k(\a_k)=\a_{k-1}\oplus\b_k\oplus\c_k,
    \end{equation*}
    where the equalities follow by \cref{Lmm: Weyl automorphisms} and note that these are subalgebras by \cref{Lmm: Commutators}.
    We prove by induction on $k$ for $1\le k\le n$ that $\restr{\rho}{KA_k}$ is injective. By \cref{Lmm: V is faithful} this is equivalent to showing that $V_0$ is faithful as a $KA_k$-module. The case $k=1$ follows then from \cref{Prop: Cartan faithful} as $\a_1=\a_0$. Now fix $2\le k\le n$ and suppose $\restr{\rho}{KA_{k-1}}$ is injective. By applying the automorphism $\sigma_k$, it is then enough to show that $KA_k^{\pm}$ acts faithfully on $V_0$, which we do in two steps. The following pictures provide an illustration of this argument for $\sp_4$.
    
\begin{figure}[!htb]
    \centering
    \begin{subfigure}{0.33\linewidth}
        \centering
        \includegraphics[page=7,valign=c]{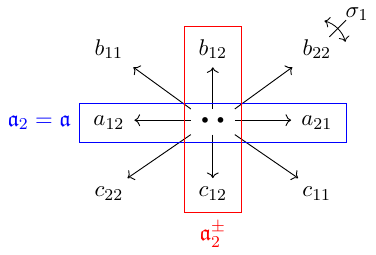}
        \caption{Action of $\sigma_1$}
    \end{subfigure}
    \begin{subfigure}{0.66\linewidth}
        \centering 
        \includegraphics[page=8,raise=-4338009sp]{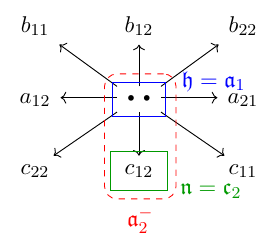} $\to$
        \includegraphics[page=9,raise=-4338009sp]{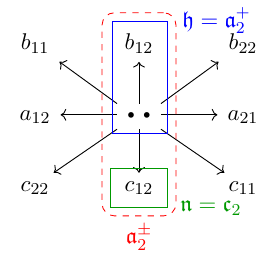} 
        \caption{Using \cref{Prop: Gluing II} to show $\restr{\rho}{KA_k^{\pm}}$ is injective}
    \end{subfigure}
\end{figure}
    \noindent \textbf{Step 1:} The restriction $\restr{\rho}{KA_k^+}$ is injective.
    
    We apply \cref{Prop: Gluing I} with
    \begin{align*}
        \n&=\c_k, & \h&=\a_{k-1}, & \V&=\{X_k^{r_k}\cdots X_n^{r_n}: r_k\in\Z, r_{k+1}, \ldots, r_n\in\N_0\} \subseteq V
    \end{align*}
    and note that $\n\oplus\h = \a_k^-$.  For $v=X_k^{r_k}\cdots X_n^{r_n} \in \mathcal{V}$ let
    $$W_v\coloneqq\langle X_1^{\alpha_1}\cdots X_n^{\alpha_n}: \alpha\in\N_0^n, \, \alpha_k=\alpha_1+\cdots+\alpha_{k-1}+r_k, \, \alpha_i=r_i \text{ for } i>k\rangle_K \subseteq \hyp{n}\cdot X_k^{r_k}\cdots X_n^{r_n},$$ where the inclusion follows since
    $$X_1^{\alpha_1}\cdots X_n^{\alpha_n}= (-1)^{\alpha_1 + \cdots + \alpha_{k-1} } c_{1k}^{\alpha_1}\cdots c_{k-1, k}^{\alpha_{k-1}}\cdot X_k^{r_k}\cdots X_n^{r_n}$$ for all $\alpha\in\N_0^n$ with $\alpha_k=\alpha_1+\cdots+\alpha_{k-1}+r_k$ and $\alpha_i=r_i$ for $i>k$.
    These are $\aff{h}$-submodules since $\a_{k-1}$ acts by homogeneous elements of degree 0 that cannot change the individual degrees in any of $X_k,\ldots,X_n$ nor the total degree in $X_1,\ldots,X_{k-1}$. Then $W=K[X_1, \ldots, X_n]$ and we have 
    \begin{itemize}
        \item $RH$ acts locally finitely on $W$ since for $f\in W$, $$RH\cdot f \subseteq \left\{ g \in K[X_1,\ldots,X_n] : \deg g=\deg f,\, v_p(g) \geq v_p(f) \right\},$$
        which is finitely generated as an $R$-module, hence so is $RH\cdot f$ as $R$ is noetherian.
        \item $KH$ acts faithfully on $K[X_1, \ldots, X_n]$ since by the induction hypothesis $\ker\restr{\rho}{KH}=0$ and $\rho(KH)$ acts faithfully on $K[X_1, \ldots, X_n]$ by \cref{Lmm: V is faithful}.
        \item The multiplication map is injective by \cref{Prop: Multiplication map} (or alternatively by \cref{Lmm: Zero annihilator}).
    \end{itemize}
    So $V_0$ is faithful as a $KA_k^-$-module, and applying the automorphism $\sigma$ we get that $KA_k^+$ acts faithfully on $V_0$ as well. Then $\restr{\rho}{KA_k^+}$ is injective.

    \noindent \textbf{Step 2:} $V_0$ is a faithful $KA_k^{\pm}$-module.
    
    Now we apply \cref{Prop: Gluing I} again, replacing $\h = \a_{k-1}$ with $\h = \a_k^+ = \a_{k-1} \oplus \b_k$ and keeping everything else unchanged. That is,
    \begin{align*}
        \n&=\c_k, & \h&=\a_k^+ & \V&=\{X_k^{r_k}\cdots X_n^{r_n}: r_k\in\Z, r_{k+1}, \ldots, r_n\in\N_0\} \subseteq V.
    \end{align*}
    and $$W_v\coloneqq\langle X_1^{\alpha_1}\cdots X_n^{\alpha_n}: \alpha\in\N_0^n, \, \alpha_k=\alpha_1+\cdots+\alpha_{k-1}+r_k, \, \alpha_i=r_i \text{ for } i>k\rangle_K \subseteq \hyp{n}\cdot X_k^{r_k}\cdots X_n^{r_n}.$$
    Note that $\n\oplus\h=\a_k^{\pm}$ and that $W_v$ are still $\aff{h}$-submodules since the action of $\b_k$ keeps $\alpha_k-(\alpha_1+\cdots+\alpha_{k-1})$ invariant. Then we have
    \begin{itemize}
        \item $RH$ still acts locally finitely on $W$ since for $f\in W$, $$RH\cdot f \subseteq \left\{ g \in K[X_1,\ldots,X_n] : \deg g\le \deg f,\, v_p(g) \geq v_p(f) \right\},$$
        which is finitely generated as an $R$-module, hence so is $RH\cdot f$ as $R$ is noetherian.
        \item $KH$ acts faithfully on $K[X_1, \ldots, X_n]$ since by the previous step $\restr{\rho}{KH}$ is injective, and $\rho(KH)$ acts faithfully on $K[X_1, \ldots, X_n]$ by \cref{Lmm: V is faithful}.
        \item The multiplication map is injective by \cref{Prop: Multiplication map} (or alternatively by \cref{Lmm: Zero annihilator}).
    \end{itemize}
    So $V_0$ is faithful as a $KA_k^{\pm}$-module, completing the induction.
\end{proof}

\subsection{Final Gluing} \label{Subsec: Gluing}
We can now glue all three subalgebras together to prove the main result.
\begin{theorem}\label{Thm: Main}
    $\restr{\rho}{KG}$ is injective. 
\end{theorem}
\begin{proof}
    Let $\s\coloneqq\a\oplus\c$ and $\t\coloneqq\a\oplus\b$ and note that these are subalgebras by \cref{Lmm: Commutators}. The following pictures provide an illustration of this argument for $\sp_4$.
    \[
        \includegraphics[page=10,raise=-3986129sp]{images.pdf} \qquad \to \qquad\quad 
        \includegraphics[page=11,raise=-4121799sp]{images.pdf}
    \]
    We first show that $V_0$ is faithful as a $KS$-module by applying \cref{Prop: Gluing I} with 
    \begin{gather*}
       \n=\c, \qquad \h=\a, \qquad  \V=\{1, X_1^{-1}\} \subseteq V, \\
        W_1=\langle X_1^{\alpha_1}\ldots X_n^{\alpha_n}: \alpha\in\N_0^n, \abs{\alpha} \text{ is even}\rangle_K \subseteq \hyp{n} \cdot 1, \\
        W_{X_1^{-1}}=\langle X_1^{\alpha_1}\cdots X_n^{\alpha_n}: \alpha\in\N_0^n, \abs{\alpha}\text{ is odd}\rangle_K \subseteq \hyp{n} \cdot X_1^{-1},
    \end{gather*}
    where the inclusions follow as $\{X_iX_j : 1 \leq i,j \leq n\}$ generates $W_1$.
    Note that $\n\oplus\h=\s$ and that $W_1$ and $W_{X_1^{-1}}$ are $\aff{h}$-modules since the elements of $\a$ act by homogeneous operators of degree $0$. Then $W=K[X_1, \ldots, X_n]$ and we have
    \begin{itemize}
        \item $RH$ acts locally finitely on $W$ since for $f\in W$, $$RH\cdot f \subseteq \left\{ g \in K[X_1,\ldots,X_n] : \deg g=\deg f,\, v_p(g) \geq v_p(f) \right\},$$
        which is finitely generated as an $R$-module, hence so is $RH\cdot f$ as $R$ is noetherian.
        \item $KH$ acts faithfully on $K[X_1, \ldots, X_n]$ since by \cref{Prop: KA faithful} $\restr{\rho}{KH}$ is injective, and $\rho(KH)$ acts faithfully on $K[X_1, \ldots, X_n]$ by \cref{Lmm: V is faithful}.
        \item The multiplication map is injective by \cref{Prop: Multiplication map}.
    \end{itemize}
    Then $V_0$ is faithful as a $KS$-module and by applying the automorphism $\sigma$ it is also faithful as a  $KT$-module so $\restr{\rho}{KT}$ is injective. 
        
    We finally apply \cref{Prop: Gluing II} again, only changing $\h$. In particular,
    \begin{gather*}
       \n=\c, \qquad \h=\t, \qquad  \V=\{1, X_1^{-1}\} \subseteq V, \\
        W_1=\langle X_1^{\alpha_1}\ldots X_n^{\alpha_n}: \alpha\in\N_0^n, \abs{\alpha} \text{ is even}\rangle_K \subseteq \hyp{n} \cdot 1, \\
        W_{X_1^{-1}}=\langle X_1^{\alpha_1}\cdots X_n^{\alpha_n}: \alpha\in\N_0^n, \abs{\alpha}\text{ is odd}\rangle_K \subseteq \hyp{n} \cdot X_1^{-1},
    \end{gather*}
    Note that $\n\oplus\h=\g$ and $W_1$ and $W_{X_1^{-1}}$ are still $\aff{t}$-submodules since the elements of $\b$ act by homogeneous operators of even non-positive degree. Then:
    \begin{itemize}
        \item $RH$ acts locally finitely on $W = K[X_1,\ldots,X_n]$ since for $f\in W$, $$RH\cdot f \subseteq \left\{ g \in K[X_1,\ldots,X_n] : \deg g\le \deg f,\, v_p(g) \geq v_p(f) \right\},$$
        which is finitely generated as an $R$-module, hence so is $RH\cdot f$ as $R$ is noetherian.
        \item $KH$ acts faithfully on $K[X_1, \ldots, X_n]$ since $\restr{\rho}{KH} = \restr{\rho}{KT}$ is injective, and $\rho(KH)$ acts faithfully on $K[X_1, \ldots, X_n]$ by \cref{Lmm: V is faithful}.
        \item The multiplication map is injective by \cref{Prop: Multiplication map}.
    \end{itemize}
    So $V_0$ is faithful as a $KG$-module and hence $\restr{\rho}{KG}$ is injective.
\end{proof}

\section{Abelian subalgebras}\label{Sec: Abelian subalgebras}
In this section, let $G$ be any uniform pro-$p$ group, and $H \leq G$ a torsion-free abelian pro-$p$ group. We give conditions for maps out of $KH$ with infinite-dimensional image to be injective, first in terms of closed subgroups of $\Aut(H)$, and then in terms of the normaliser of $\h$ in $\g$.

\subsection{Invariant ideals in commutative algebras}
Let $H$ have topological generating set $H_1,\ldots,H_d$. We equip $KH$ with the $p$-adic filtration 
$$w_p\left( \sum_{\alpha\in\N_0^d}\lambda_{\alpha}\left(\mathbf{H}-1 \right)^{\alpha} \right)=\inf_{\alpha\in\N_0^d}v(\lambda_{\alpha})$$
where $(\mathbf{H}-1)^\alpha = (H_1-1)^{\alpha_1}\cdots (H_d-1)^{\alpha_d}$ and the $\lambda_{\alpha}\in K$ are uniformly bounded. Note that $w_p$ is complete and separated with $\gr KH \cong kH[s,s^{-1}]$, where $s$ is the image of the uniformiser of $K$.

We endow $\Aut(H)$ with the congruence topology as defined in \cite[\textsection 5.2]{ddms2003analytic}, and let $i: \Gamma \hookrightarrow \Aut(H)$ be a closed subgroup such that the action of $\Gamma$ on $H$ is uniform, namely $$[\Gamma, H] \coloneqq  \{\gamma(h)h^{-1} : \gamma \in \Gamma,\,h \in H\} \subseteq H^p.$$ In particular, $\Gamma$ is then isomorphic to a closed subgroup of the first congruence subgroup $\Gamma_1 \coloneqq 1 + pM_d(\Z_p)$ of $\GL_d(\Z_p)$, which is uniform by \cite[Theorem 5.2]{ddms2003analytic}. Moreover, $\Gamma$ is a finitely generated pro-$p$ group by \cite[Theorem 3.8]{ddms2003analytic}, and every $\Gamma^l \coloneqq \left\langle\gamma^l : \gamma \in \Gamma \right\rangle \leq \Aut(H)$ is also a closed subgroup by \cite{martinez1994power}.

As in \cite[Section 4.2]{ardakovwadsley} we have that $\L(H)$ is naturally an $\L(\Gamma)$-module and we will further assume it is irreducible.

We extend the action of $\Gamma$ first to $RH$ by linearity and then to $KH$, as the $\Gamma$-action does not affect the $w_p$-filtration. Similarly, $\Gamma$ acts on the local ring $kH$, and is exactly chosen so that \cite[Corollary 8.1]{ardakov2012prime} gives

\begin{lemma}\label{Lmm: Invariant ideals in kH}
    For $l \in \N$, the only $\Gamma^l$-invariant prime ideals of $kH$ are 0 and the maximal ideal $\m = (H_1-1,\ldots,H_d-1)$.
\end{lemma}
This uses the observation that $\L(\Gamma^l) = \L(\Gamma)$, which follows from $L_{\Gamma^l}=lL_{\Gamma}$ (see \cite[Theorem 7.4]{CompactPadic}), and recalling that 
$$\L(\Gamma^l) = L_{\Gamma^l} \otimes \Q_p  = l L_\Gamma \otimes \Q_p = \L(\Gamma),$$
where the first and last equalities are the definition of $\L$, and the third follows from the proof of \cite[Theorem 9.8]{ddms2003analytic}. This maximality condition lifts to a condition on $\Gamma$-invariant ideals in $KH$.

\begin{proposition}\label{Prop: Invariant ideals in KH}
    Let $I$ be an non-zero, $\Gamma$-invariant ideal in $KH$. Then $I$ has finite $K$-codimension.
\end{proposition}
\begin{proof}
    If $I = KH$ this is trivial, so suppose not. Consider the $\gr K$-module $\gr KH$, and the non-zero, proper, $\Gamma$-invariant, graded submodule $\gr I$. Identifying $\gr KH \cong kH[s,s^{-1}]$, we see that $s$ is a unit, and so each graded component of $\gr I$ is equal. Thus we may write $\gr I = J[s,s^{-1}]$ for some ideal $J \subseteq kH$.
    
    As $kH$ is noetherian, $\rad(J) = \bigcap_{i=1}^m \p_i$, where $\p_1,\ldots,\p_m$ are minimal prime ideals of $kH$. As $\rad(J)$ is also $\Gamma$-invariant and $\Gamma$ acts by automorphisms, each element of $\Gamma$ permutes the $\p_i$, so we can find an $l \geq 1$ such that $\Gamma^l$ fixes all of the $\p_i$. So by \cref{Lmm: Invariant ideals in kH}, $\rad(J) = \m$.
    
    Thus we can find an $r \geq 0$ such that $\m^r \subseteq J$, and since $\faktor{kH}{\m^r}$ is finite-dimensional over $k$ so is $\faktor{kH}{J}$. Now note that $$\faktor{\gr KH}{\gr I} \cong \faktor{kH[s,s^{-1}]}{J[s,s^{-1}]} \cong \left(\faktor{kH}{J}\right)[s,s^{-1}]$$ so $\faktor{\gr KH}{\gr I}$ is finite-dimensional over $\gr K \cong k[s,s^{-1}]$. By \cite[Theorem I.4.2.4]{huishi1996zariskian} we have $\faktor{\gr KH}{\gr I} \cong \gr \left(\faktor{KH}{I} \right)$   
    so $\gr\left(\faktor{KH}{I} \right)$ is generated over $\gr K$ by finitely many elements of degree 0. Finally, by \cite[Theorem I.5.7]{huishi1996zariskian}, $\faktor{KH}{I}$ has finite $K$-dimension.
\end{proof}

This immediately gives the following corollary.

\begin{corollary}\label{Cor: General invariance}
    Let $T$ be a filtered $K$-algebra, and let $\psi: KH \to T$ a filtered $K$-algebra homomorphism such that $\psi(KH)$ is not finite $K$-dimensional. If $\ker \psi$ is $\Gamma$-invariant, then $\psi$ is injective. 
\end{corollary}

\subsection{Abelian subalgebras}

A key example of such automorphism groups $\Gamma$ comes from conjugation. %Suppose now that $H$ is subgroup of a uniform pro-$p$ group $G$, and let $\h, \g$ be their associated $\Z_p$-Lie algebras.
Let $\n \subseteq \g$ be a subalgebra contained in the normaliser $\{x \in \g : [x,y] \in \h \text{ for } y \in \h \}$ of $\h$, and $N$ its associated pro-$p$ group. We assume that $\g/\n$ is torsion-free as a $\Z_p$-module so that by \cite[Proposition 7.15]{ddms2003analytic}, $N \leq_c G$ is a closed uniform subgroup.

\begin{theorem}\label{Prop: Invariance from conjugation}
    Let $T$ be a filtered $K$-algebra, and let $\psi: KG \to T$ be a filtered $K$-algebra homomorphism such that $\psi(KH)$ is not finite $K$-dimensional. If $\L(H)$ is an irreducible $\L(N)$-module, then $\restr{\psi}{KH}$ is injective.
\end{theorem}
\begin{proof}
    We first show that the conjugation homomorphism $\varphi: N \to \Aut(H)$ by $x \mapsto (\varphi_x: y\mapsto xyx^{-1})$ is continuous. To this end, suppose a net $(x_{\lambda})_{\lambda \in \Lambda}$ converges to $1_N$ in $N$. It suffices to prove $\varphi_{x_{\lambda}}$ converges to $\id_H$ in the congruence topology; that is, that for any open normal subgroup $H' \lhd_o H$, there is a $\mu_{H'} \in \Lambda$ such that $[\varphi_{x_\lambda}, H] \subseteq H'$ for $\lambda \geq \mu_{H'}$.

    Fix some $H' \lhd_o H$. For any $y \in H$, $\varphi_{x_\lambda}(y) y^{-1} = [x_\lambda,y]$ converges to $1_H$ in $H$, so there is a $\mu_y \in \Lambda$ such that $[x_\lambda, y] \in H'$ for $\lambda \geq \mu_y$. If $h \in H'$, we also see that $[x_\lambda, hy] \in H'$ by normality, so $[x_\lambda, g] \in H'$ whenever $g \in H'y$, $\lambda \geq \mu_y$. As $H$ is profinite, $H'$ has finite index in $H$, so it is enough to take $\mu_{N'} \geq \mu_{y_i}$, $1 \leq i \leq m$, for some finite collection $y_1,\ldots,y_m$ of coset representatives of $H/H'$.

    We check that the action of $\varphi(N)$ on $H$ satisfies the hypotheses of \cref{Cor: General invariance}.
    
    First, note that $[\n,\h] \leq  \h \cap p \g = p\h$ as $\g$ is a powerful Lie algebra and $\g/\h$ is torsion-free. Then the Campbell-Baker-Hausdorff formula \cite[Chapter 3]{Hall2003} guarantees that $[\varphi(N), H] = [N,H] \leq H^p$, so the action of $\varphi(N)$ on $H$ is uniform. The map  $\L(\varphi) : \L(N) \to \L(\varphi(N))$ is surjective by \cite[\textsection 9, Ex. 7]{ddms2003analytic}, so $\L(H)$ is irreducible as an $\L(\varphi(N))$-module. Finally, $\ker \psi$ is a two-sided ideal of $KG$ and the conjugation action of $N$ on $G$ fixes $H$, so both $\ker \psi$ and $KH$ are $N$-invariant. Thus $\ker \restr{\psi}{KH} = KH \cap \ker \psi$ is $N$-invariant.
    
    Hence we can apply \cref{Cor: General invariance} with $\Gamma = \varphi(N)$.
\end{proof}

We apply \cref{Prop: Invariance from conjugation} with $G = \exp(p \sp_{2n}(\Z_p))$, $H = C$, $N=A$ and $\psi=\rho$ as in \cref{Sec: Introduction} to obtain the injectivity of the restrictions $\restr{\rho}{KB}$ and $\restr{\rho}{KB}$. Note then $\g = p \sp_{2n}(\Z_p)$, $\h = p\c$, $\n = p\c$ in the notation of \cref{Sec: Introduction}.
 
\begin{corollary}\label{Prop: KC faithful}
    $\restr{\rho}{KC}$ and $\restr{\rho}{KB}$ are injective.
\end{corollary}
\begin{proof} 
    $KC$ acts on the Tate algebra $K\langle X_1,X_2,...X_n\rangle$ via $\rho$ with $$(C_{12}-1)\cdot X_1^{n_1}X_2^{n_2} = \left(\exp \left(-px_1x_2 \right) -1\right)\cdot X_1^{n_1}X_2^{n_2} = -pX_1^{n_1+1}X_2^{n_2+1} +\frac{p^2}{2}X_1^{n_1+2}X_2^{n_2+2} + \cdots$$ which shifts the total trailing degree by $2$, and hence the actions of $(C_{12}-1)^n$ for $n\in\N$ are linearly independent. Then the image of $\rho(KC)$ is infinite dimensional.
    
    By \cref{Lmm: Commutators}, $p\a$ is contained in the normaliser of $p\c$. $p\mathfrak{a}\otimes_{\Z_p}\Bar{\Q}_p$ contains a copy of $\mathfrak{sl}_{n}(\Bar{\Q}_p)$, so we can view $p\mathfrak{c}\otimes_{\Z_p}\Bar{\Q_p}$ as a $\mathfrak{sl}_{n}(\Bar{\Q}_p)$-module by restriction of scalars. Now by \cref{Lmm: Commutators} $p\c$ has a unique highest weight vector $pc_{nn}$ up to scalars for the choice of positive roots corresponding to $a_{ij}$ with $i<j$, hence it is simple. Then $\L(C) = p\c\otimes_{\Z_p}\Q_p$ is simple for $\sl_{n}(\Q_p)$ and therefore also for $\L(A) = p\a\otimes_{\Z_p}\Q_p$.

    The argument for $H=B$ is completely analogous.
\end{proof}

\renewcommand\bibname{\scriptsize References}
\printbibliography[title=References]
\end{document}